\documentclass{lmns}

\usepackage{amsthm}
\usepackage{amsmath}
\usepackage{amsfonts}
\usepackage{amssymb}
\usepackage{times}
\usepackage{graphicx}

\theoremstyle{definition}

\newtheorem{rem}{Remark}

\theoremstyle{plain}

\input{tcilatex}
\def\tilde{\widetilde}

\def\N{\mathbb N}

\def\L{{\mathcal L}}

\def\H{{\mathcal H}}
\def\I{{\mathcal I}}

\def\K{{\mathcal K}}

\begin{document}

\lmnsVolume{33}{2}{2013}{53}{doi:10.14510/lm-ns.v33i2.xx} 
\lmnsTitle[Conditional expectations, traces. angles between
spaces]{Conditional expectations, traces, angles between spaces and
Representations of the Hecke algebras} 
% the author names format is: First_name Family_name (eg., John Miller)
\lmnsAuthors[Florin R\u adulescu]{Florin R\u adulescu${}^*$} %
\lmnsDedication{Dedicated to Professor  C.T. Ionescu Tulcea, on the occasion
of his 90th anniversary}

\lmnsAbstract{In this paper we extend the results in [Ra] on the representation of the Hecke algebra, determined by the matrix coefficients of a projective, unitary representation, in the discrete series of representations  of the ambient group, to a more general, vector valued case.
This method is used to analyze the traces of the Hecke operators.

We construct  representations of the Hecke algebra  of a group $G$, relative to an almost normal subgroup $\Gamma$,  into the von Neumann algebra of the group $G$, tensor matrices. The representations we obtain  are a lifting  of  the Hecke operators to this larger algebra. By summing up the coefficients  of the terms in the representation one obtains the classical Hecke operators.

 These representations were used in the scalar case in [Ra], to find an alternative representation of the Hecke operators on Maass forms, and hence to reformulate the Ramanujan Petersson conjectures as a problem on the angle (see e.g. A. Connes's paper [Co] on the generalization of CKM matrix) between two subalgebras  of the von Neumann algebra of the group $G$: the image of the representation of the Hecke algebra and the algebra of the almost normal subgroup.}
\lmnsKeywords{Conditional expectations, Traces, Hecke Algebra,
Ramanujan-Petersson conjecture} \lmnsMSC{11F72, 46L65} 
\lmnsContact{Florin R\u
adulescu}{Dipartimento di Matematica, Universita degli Studi di Roma ``Tor
Vergata'', Italy \\ and \\Institutul de Matematic\u{a} ''Simion Stoilow'',
Academia Rom\^{a}n\u{a}, 014700 Bucharest,
Romania}{radulesc@mat.uniroma2.it} \lmnsMaketitle

\renewcommand{\thefootnote}{} \footnotetext{${}^{\ast }$ 
Supported in part by
PRIN-MIUR, and by PN-II-ID-PCE-2012-4-0201}

\section{Introduction}

Let $G$ be a countable discrete group and $\Gamma \subseteq G$ be an almost
subgroup. Assume $\pi$ is a unitary representation of $G$ into the unitary
group of a separable Hilbert space $H$
and  assume that $\pi|_{\Gamma}$ is a multiple of the left regular
representation $\lambda_{\Gamma}$ of $\Gamma$ on $l^2(\Gamma)$. For
simplicity, throughout the paper we assume that the groups $\Gamma$ and $G$
have infinite conjugacy classes, and hence  it follows  that the associated von Neumann
algebras are factors, and thus have a unique trace denoted by $\tau$. The
Murray von Neumann theory of dimension (see e.g. [Ta], [GHJ]) associates to
the representation $\pi$ of the group $\Gamma$, a continuous dimension number  $t = 
\mathrm{dim}_{\pi(\Gamma)^{\prime\prime}}H = \mathrm{dim}_{\Gamma}H \in (0,
\infty]$. Here, by $\pi(\Gamma)^{\prime\prime}$ we denote, as customary in
von Neumann algebras, the von Neumann algebra generated by $\pi(\Gamma)
\subseteq B(H)$. In the paper [Ra] we treated the case $\mathrm{dim}%
_{\Gamma}H = 1$. (We are deeply indebted to H. Moscovici for suggesting to
investigate the general case of arbitrary dimension). If $t$ is an integer
or $\infty$ on infinity, the hypothesis on the representation of $G$
corresponds to the existence of a Hilbert subspace $L \subseteq H$, such
that $H \cong l^2(\Gamma) \otimes L$, and $\pi|_\Gamma \cong \lambda_{\Gamma} \otimes 
\mathrm{Id}_{L}$.

To this data we associate a representation $S$ of the algebra of Hecke
operators into the reduced von Neumann algebra of the right regular
representation of $G$, $R(G)_t$ ([MvN]). (If $t$ is an integer, the algebra $%
R(G)_t$ is $R(G) \otimes B(K)$, where $K$ is any Hilbert space of finite dimension $%
\mathrm{dim}_{\mathbb{C}}K = t$), otherwise $R(G)_t = {p}(R(G) \otimes
B(L))p $, where $p$ is a projection in $\mathcal{R}(\Gamma) \otimes B(L)$ of
trace $\tau(p)= t$, and $L$ is an arbitrary infinite dimensional Hilbert
space.
In the case $\mathrm{dim}_{\mathbb{C}}L = \infty$, the representation takes
values into the formal ring of infinite series in $G$ with coefficients in $%
B(L)$.

Let $\mathcal{H}_0 = \mathbb{C}(\Gamma\setminus G/\Gamma)$ be the $\mathbb{C%
}$ - algebra of double cosets of $\Gamma$ in $G$. Let $\mathcal{H}$ be the reduced von Neumann
- Hecke algebra of $\Gamma \subseteq G$, acting on $l^2(\Gamma/G)$. Recall that this the the weak
operator topology closure of $\mathcal{H}_0$, viewed as a subalgebra of $%
B(l^2(\Gamma/G))$, (see [BC] ).

We generalize the results in [Ra], by proving that also in this more general
case,  the angle (see Theorem \ref{t}) between the algebra $S(\mathcal{H}%
)\subseteq R(G)_{t}$ and the algebra $R(\Gamma )_{t}$, determines the
Ramanujan - Petersson behavior of the representation $\Psi $ of the
(algebraic) Hecke algebra $\mathcal{H}_{0}=\mathbb{C}(\Gamma \setminus
G/\Gamma )$ (the $\mathbb{C}$ - algebra of double cosets) on $\pi (\Gamma
)^{\prime }\subseteq B(H)$. The representation $\Psi$ is simply the action of the Hecke algebra on $%
\Gamma$ invariant elements in $B(H)$. The space of the $\Gamma$- invariant
elements is exactly the algebra $\pi(\Gamma)^{\prime}$. Here $G$ acts on $B(H)$ by $\mathrm{Ad}\pi (g),$ 
$g\in G$ and $\pi (\Gamma )^{\prime }$ is the commutant algebra. By the
Murray von Neumann dimension theory the commutant algebra is isomorphic to $%
R(G)_{t}$.  

As proved in [Ra] an essential ingredient to prove that the
representations of the Hecke algebra may be used to construct Hecke
operators on $\pi (\Gamma )^{\prime }$ (which is then equivalent to the problem
of determining the action of Hecke operators on Maass forms), is the fact
that the representation of the Hecke algebra may be extended naturally to an
operator system containing the Hecke algebra (conform Definition \ref{os}).

If $\Gamma =PSL_{2}(\mathbb{Z})$, $G=PGL_{2}(\mathbb{Z}_{\frac{1}{p}})$, $%
\pi $ is the restriction to $G$ of a representation (which could also be projective) in the discrete series
of $PGL_{2}(\mathbb{R})$. Then, through Berezin
symbol map (Cf. [Be]), one recaptures from the representation $\Psi $ the
classical Hecke operators on Maass forms.

The explanation of the fact that we are able to recapture the action of
classical Hecke operators on $\Gamma$ invariant vectors is as follows:
Starting with the representation $S$ of the Hecke algebra into ${p}(R(G)
\otimes B(L))p$ one defines the densely defined, $\ast$-algebra morphism, $%
\widetilde{\varepsilon} = \varepsilon \otimes \mathrm{Id}$ from $R(G)
\otimes B(L)$ into $B(L)$ by letting $\varepsilon$ be the character of $R(G)$
(that is $\varepsilon(\sum a_g\rho_g) = \sum a_g$).

If $S(\mathcal{H}_0)$ is contained in the domain of $(\varepsilon \otimes 
\mathrm{Id})$, then $\widetilde{\varepsilon} \circ S$ is a representation of 
$S$ onto $B(L)$ (in fact $\varepsilon(p)L$). Identifying $\varepsilon(p)L$
with the $\Gamma$-invariant vectors in the representation $\pi$, we obtain in this way
the representation of the classical operators Hecke operators on $\Gamma$%
-invariant vectors.

In particular applying the above construction to the representation $\pi
\otimes \overline{\pi}$ acting on $H \otimes \overline{H}$, one obtains (as it was also explained in [Ra]) the
classical Hecke operators on the Hilbert space of $\Gamma$ - invariant
vectors in $H \otimes \overline{H}$. This latest Hilbert space is
canonically identified to the Hilbert space $L^2(\pi(\Gamma)^{\prime}, \tau)$
associated to the von Neumann algebra $\{ \pi(\Gamma) \}^{\prime}$.

We use the following construction of the Hilbert space of the associated $%
\Gamma$ - invariant vectors. We assume the representation $\pi$,  denoted in the following
by $\pi_0$, acts on the Hilbert space $H_0$. To define a Hilbert space of $%
\Gamma$ - invariant vectors we use a "scale for $\Gamma$-invariant vectors" (as it is done in the
Petersson scalar product).

Thus, we assume that $H_0$  is contained in a larger Hilbert space $ H$, and we assume that there exists  a representation  $\pi$  of $G$ into the unitary group of $H$,
 such that $\pi$ invariates $H_0$, and such that $$%
\pi(g)|_{H_0} = \pi_0(g),\   g\in G.$$

We assume $L \subseteq H$, is a subspace (which will be the scale mentioned above) such that $\pi|_{\Gamma} \cong
\lambda_{\Gamma} \otimes \mathrm{Id}_L$ (thus $H \cong l^2(\Gamma) \otimes L$%
). We define the Hilbert space of $\Gamma$-invariant vectors (as in the
Petersson scalar product) by letting $H_0^{\Gamma} = H_0^{\Gamma}(\pi, L)$
be the subspace of densely defined forms, on $H_0$, such that $\langle v, v
\rangle_{H_0^{\Gamma}} = \langle P_Lv, v \rangle$ is finite (that is such that
the densely defined map on $H$ associating to a vector $w$ in the domain the
value $\langle P_Lv, w \rangle = \langle v, P_Lw \rangle$ extends to a
continuous form on $H$).

Then $\widetilde{\varepsilon }\circ S$ is the classical Hecke operator
representation on $H_{0}^{\Gamma }(\pi ,L)$ (see Section 3). We obtain in this way another representation of the Hecke operators on $\Gamma $
invariant vectors associated to the action of $\Gamma $ on $H_{0}$. The Hecke operators are represented as
operators acting on a subspace of $L$. These representation can be used then to
analyze the traces of the Hecke operators.

Finally, in the last section we determine a precise formula for the absolute
value of the matrix coefficients of the 13-th element $\pi _{13}$ in the
discrete series of $PSL_{2}(\mathbb{R})$, when restricted to $G=PGL_{2}(%
\mathbb{Z}[\frac{1}{p}])$.

This proves that this coefficients are absolutely summable on cosets of $%
\Gamma = PSL_2(\mathbb{Z})$ (a fact that is needed to prove that the map $\varepsilon$ defined above is densely defined).

\vskip6pt

\section{The representation of the Hecke algebra}

In this section, we introduce a representation of the Hecke algebra ${\mathcal{H}}_{0}$ associated to an
almost normal subgroup $\Gamma $ of a discrete group $G$. This representation of the Hecke algebra   is canonically
associated to a projective, unitary representation $\pi $ of the larger
group $G$.

 Let $\pi : G \to B(H)$ be a projective, unitary
representation of the group $G$ into the unitary operators on the Hilbert
space $H$.
We assume that $\pi |_{\Gamma }$ is multiple representation of the left
regular representation of the group $\Gamma $ (which, as was observed in
[Ra]) implies that the associated Hecke algebra of double cosets ${\mathcal{H}}_{0}=\mathbb{C}%
(\Gamma \setminus G/\Gamma )$ is unimodular, i.e. $[\Gamma :\Gamma _{\sigma
}]=[\Gamma :\Gamma _{\sigma ^{-1}}]$, with $\Gamma _{\sigma }=\sigma \Gamma
\sigma ^{-1}\cap \Gamma $, for all $\sigma \in G$.

In addition throught this section we assume that there exists a Hilbert
subspace $L$ of $H$ such that $L$ is generating, and $\Gamma $ - wandering
for $H$ (that is $\gamma L$ is orthogonal to $L$ for all $\gamma $ in $%
\Gamma \setminus \{e\}$, and the closure of the linear span $\overline{Sp(%
\mathop{\cup}\limits_{\gamma \in \Gamma }\gamma L)}$ is equal to $H$). Here
by $e$ we denote the identity element of the group $G$.

Then $H$ is identified to the Hilbert space $l^2(\Gamma) \otimes L$, and $L$
is identified to $\{e\} \otimes L$, and the representation $\pi$ is
identified to the left regular representation $\lambda_{\Gamma}$ of $\Gamma$
on $l^2(\Gamma)$, tensor the identity of $L$.

In the sequel we denote the right regular representation of $\Gamma$
(respectively $G$) by $\rho_{\Gamma}$ (respectively $\rho_G$) on $%
l^2(\Gamma) $ (respectively $l^2(G)$). When no confusion is possible we
simply denote this by $\rho$. By $\widehat{R(G) \otimes B(L)}$ we denote
the algebra of formal series

\centerline{$\mathop{\sum}\limits_{g \in G} \rho(g) \otimes A_g$}

\noindent with coefficients $A_g \in B(L)$, $g \in G$, that have the
additional property that for every $l_1, l_2 \in L$, the coefficient

\centerline{$\langle( \mathop{\sum}\limits_{g \in G} \rho(g) \otimes A_g)
(\delta_e \otimes l_1), \delta_e \otimes l_2 \rangle$}

\noindent defined by

\centerline{$\mathop{\sum}\limits_{g \in G} \rho(g) \langle A_g l_1, l_2
\rangle$}

\noindent is an element of $R(G)$, the von Neumann algebra of right, bounded
convolutors on $l^2(G)$.

With the above definitions, we consider the following linear operator system
(see [Pi] for the definition of an operator system):

\vskip6pt

\begin{definition}\label{os} 
By ${\mathcal{K}}_0^{\ast}{\mathcal{K}}_0$ we denote the tensor
product $\mathbb{C}(G \setminus \Gamma) \mathop{\otimes}\limits_{{\mathcal{H}%
}_0} \mathbb{C}(\Gamma \setminus G)$. Here ${\mathcal{H}}_0 = \mathbb{C}%
(\Gamma \setminus \Gamma/G)$ is the Hecke algebra of double cosets and ${%
\mathcal{H}}_0$ is embedded both in ${\mathcal{K}}_0^{\ast}=\mathbb{C}(G\setminus \Gamma)$ and $%
{\mathcal{K}}_0=\mathbb{C}(\Gamma/G)$, simply by writing a double coset as a sum of right,
respectively left, cosets (see e. g. [BC]).

There is an obviously $\ast$ operation on ${\mathcal{K}}_0^{\ast}{\mathcal{K}%
}_0$ (whence the notation). Then ${\mathcal{K}}_0 = \mathbb{C}(\Gamma/G)$,
and ${\mathcal{K}}_0^{\ast} = \mathbb{C}(G\setminus \Gamma)$ are canonically
identified as subspaces of ${\mathcal{K}}_0^{\ast}{\mathcal{K}}_0$. We have
a canonical bilinear map from ${\mathcal{K}}_0^{\ast} \times {\mathcal{K}}_0$
onto ${\mathcal{K}}_0^{\ast}{\mathcal{K}}_0$. This bilinear map, when restricted to the double cosets in the
Hecke algebra ${\mathcal{H}}_0$, becomes the Hecke algebra multiplication.

As observed in ([Ra]), the operator system ${\mathcal{K}}_0^{\ast}{\mathcal{K%
}}_0$ is isomorphic to the linear span of cosets: 
\begin{equation*}
\mathbb{C} \{\sigma_1 \Gamma \sigma_2 | \sigma_1, \sigma_2 \in G\}= \mathbb{C}
\{(\sigma_1 \Gamma \sigma_1^{-1})\sigma_2 | \sigma_1, \sigma_2 \in G\},
\end{equation*}
where the bilinear operation maps $(\sigma_1\Gamma, \Gamma\sigma_2)$ into $%
\sigma_1 \Gamma \sigma_2$, for all $ \sigma_1, \sigma_2 \in G$.

There is an obviously completion of this system, which is still a $\ast$
operator system, which we denote as ${\mathcal{K}}^{\ast}{\mathcal{K}}$.
This operator system is, by definition, $l^2(G/\Gamma) \mathop{\otimes}%
\limits_{{\mathcal{H}}}l^2(\Gamma \setminus G)$, where ${\mathcal{H}}$ is
the reduced von Neumann algebra of the Hecke algebra, that is the weak
closure of ${\mathcal{H}}_0$ acting on $l^2(\Gamma \setminus G)$
(respectively on $l^2(G/\Gamma)$ by right multiplication).

A priori, there is no canonical Hilbert space structure on ${\mathcal{K}}^{\ast}{%
\mathcal{K}}$. Using the construction in this paper, we define at the
end of the paper, a canonical Hilbert structure on ${\mathcal{K}}{\mathcal{K}%
}^{\ast}$ that is compatible with the action of ${\mathcal{H}}$.  Thus  we
construct  a scalar product on cosets

\centerline{$\langle \Gamma \sigma_1 \otimes \sigma_2 \Gamma, \Gamma
\sigma_3 \otimes \sigma_4 \Gamma \rangle$}

\noindent that depends only on $\sigma_3^{-1}\Gamma \sigma_1$, $%
\sigma_2^{-1}\Gamma \sigma_3$. By cyclically rotating the
variables, the scalar product corresponds to a bilinear form on ${\mathcal{K}%
}^{\ast}{\mathcal{K}}$. (Note that in [Ra] we also constructed  a scalar
product, with positive values on the generators, such that the corresponding
bilinear form is again positive definite).
\end{definition}

With the above definitions we have the following.

\begin{proposition}
\label{rep} Let $\Gamma \subseteq G,\pi ,H,L$ and ${\mathcal{K}}_{0}^{\ast }{%
\mathcal{K}}_{0}$ as above. Then there exists a canonical representation

\centerline{$S : \mathcal{K}_0^{\ast}\mathcal{K}_0 \to \widehat{R(G) \otimes B(L)}$}

\noindent that is a $\ast$ morphism in the sense that $S(k_1)^{\ast}S(k_2) =
S(k_1^{\ast}k_2)$ where $k_1^{\ast}k_2$ is viewed as an element of ${%
\mathcal{K}}_0^{\ast}{\mathcal{K}}_0$, for all $k_1, k_2 \in {\mathcal{K}}_0$%
. In this assertion, we implicitly require that the elements $S(k_1^{\ast}) =
S(k_1)^{\ast}$, $S(k_2)$ may be multiplied in $\widehat{R(G) \otimes B(L)}$,
for all $k_1, k_2 \in {\mathcal{K}}_0$.

By restricting to ${\mathcal{H}}_0$, we obtain a representation of the Hecke
algebra. When $L$ is of finite dimension, this representation may be
extended to ${\mathcal{K}}^{\ast}{\mathcal{K}}$ (and hence to ${\mathcal{H}}$%
). Let $P_L$ be the the projection from $H$ onto $L$.

The formula for $S$ is 
\begin{equation*}
S(A) = \mathop{\sum}\limits_{\theta \in A}\rho(\theta^{-1}) \otimes P_L
\pi(\theta) P_L
\end{equation*}

\noindent where $A$ is any of the cosets $\sigma_1\Gamma, \Gamma\sigma_2$ or 
$\sigma_1\Gamma\sigma_2 = [\sigma_1\Gamma][\Gamma\sigma_2],$ for$\sigma_1,
\sigma_2 \in G$.
\end{proposition}

\vskip6pt

\begin{proof}
In the summation, with $A$ as in the statement, we have to take $\rho(\theta^{-1})$, for $\theta \in A$, since the right regular representation has the property that
$\rho (a)\rho (b)= \rho (ba)$, for $a,b \in R(G)$.

The multiplicativity of the map $S$ follows from the fact that
$$
\mathop{\sum}\limits_{\gamma \in \Gamma} P_L \pi(\theta_1\gamma)P_L \pi(\gamma^{-1}\theta_2)P_L = P_L \pi(\theta_1\theta_2)P_L.
$$

This formula is a consequence of the fact that $\sum \pi(\gamma)P_L\pi(\gamma)^{-1} = {\rm Id}$ (this is also valid for projective representations).
This is indeed the coefficient of $\rho(\theta_1\theta_2)$ in a product $S(\theta_1\Gamma)S(\Gamma\theta_2)$ for all $\theta_1, \theta_2 \in G$.

The fact that the matrix coefficients of $S$ belong indeed to $R(G)$, is observed as follows. Fix $l_1, l_2$ in $L$.

Then
$$
\mathop{\sum}\limits_{\theta \in \Gamma\sigma} \rho(\theta^{-1}) \langle \pi(\theta)l_1, l_2 \rangle = \mathop{\sum}\limits_{\gamma \in \Gamma} \rho((\gamma\sigma)^{-1}) \langle \pi(\gamma)\pi(\sigma)l_1, l_2 \rangle =
$$
$$
= \mathop{\sum}\limits_{\gamma \in \Gamma}\rho(\sigma^{-1}) \rho(\gamma^{-1})\langle \pi(\sigma)l_1, \pi(\gamma^{-1})l_2 \rangle
$$

\noindent (since $\pi(\gamma)l_2$ are orthogonal).

The fact that $S$ may be extended by linearity and continuity to $\mathcal{K}^{\ast}\mathcal{K}$ in the case that $L$ is finite dimensional is proved as follows:  in this case the matrix entries $\omega_{l_1, l_2} = \langle l_1, \cdot \rangle l_2$, $l_1, l_2 \in L$, are, when composed with $S$, the matrix entries of a representation of ${\mathcal{H}}_0$, into $\widehat{R(G) \otimes B(L)}$, which is however trace preserving, and thus extendable by continuity to ${\mathcal{H}}$. In particular, if $L$ is finite dimensional, then for all $\sigma \in G$, we have that $S(\Gamma\sigma)$, $S(\sigma\Gamma)$ belong to $R(G) \otimes B(L)$.
\end{proof}

\vskip6pt

\begin{remark}\label{exp}
We define formally 
\begin{equation*}
E_{\pi(\Gamma)^{\prime}}(A) = \mathop{\sum}\limits_{\gamma \in \Gamma}
\pi(\gamma)A\pi(\gamma)^{-1},
\end{equation*}
for $A$ in $B(H)$. If the above sum is so convergent, then this sum
represents the conditional expectation onto $\pi(\Gamma)^{\prime}$.

Moreover if $A \in {\mathcal{C}}_1(H)$ (i.e. $A$ is a trace class operator),
then the sum from the statement is simply the restriction of the state $%
\mathrm{tr}(A\cdot)$ to $\pi(\Gamma)^{\prime}$. Here the linear functional $%
\mathrm{tr}$ is the trace on the trace class operators ${\mathcal{C}}_1(H)$.

In particular if $AP_L$ belongs to ${\mathcal{C}}_1(H)$ (or even ${\mathcal{C%
}}_2(H)$) then $E_{\pi(\Gamma)^{\prime}}(AP_L)$ is clearly equal to $%
\mathop{\sum}\limits_{\gamma \in \Gamma} \rho(\gamma^{-1}) \otimes P_L
\pi(\gamma)AP_L$.

In particular, the preceding formula for 
\begin{equation*}
S(\sigma\Gamma) = \mathop{\sum}\limits_{\gamma}\rho((\sigma\gamma)^{-1})
\otimes P_L\pi(\sigma)\pi(\gamma)P_L
\end{equation*}

\noindent might be interpreted as $\rho(\sigma^{-1})E_{\pi(\Gamma)^{%
\prime}}(\pi(\sigma)P_L)$.

\ 

\ 

Let $\varepsilon$ be the densely defined character on $R(G)$ defined as the
sum of coefficients. Then we define $ = \varepsilon
\otimes \mathrm{Id}$ on $\widehat{R(G) \otimes B(L)}$ (here $\widetilde{\varepsilon}$ is densely defined) with
values in $B(L)$ by the formula 
\begin{equation*}
\widetilde{\varepsilon}(\sum \rho(g) \otimes A_g) = \mathop{\sum}\limits_{g}
A_g.
\end{equation*}

\noindent
 For $x = \sum \rho(g) \otimes A_g \in \widehat{R(G) \otimes B(L)}$%
, the condition for $x$ to be in the domain of $\widetilde{\varepsilon}$ is
that the sum $\mathop{\sum}\limits_{g} A_g$ be so-convergent.

Then $\widetilde{\varepsilon} \circ S$, if $S({\mathcal{H}}_0) \subseteq 
\mathrm{Dom\ } \widetilde{\varepsilon}$, is a representation of ${\mathcal{H}%
}_0$ into $B(L)$.
\end{remark}

\vskip6pt

\begin{remark}
Assume that $(\mathcal{X}, \nu)$ is an infinite measure space and $H = L^2(%
\mathcal{X}, \nu)$ and 
assume that $G$ acts by measure preserving transformation on $\mathcal{X}$. Let $F$ be a fundamental domain for the action $\Gamma$ on $\mathcal{X}$ (implicitly we assume here that the action of $G$ is so that such a domain exists).
Let $\pi = \pi_{\mathrm{Koop}}$ be the Koopman representation (see [Ke]) of $%
G$ on $L^2(\mathcal{X}, \nu)$ . Let $\sigma$ be any element of $G$. Let $%
\chi_F$ be the characteristic function of $F$, viewed as a projection in $%
B(L^2(\mathcal{X}, \nu))$. Then 
\begin{equation*}
S(\Gamma\sigma\Gamma) = \mathop{\sum}\limits_{\theta \in \Gamma\sigma\Gamma}
\rho(\theta^{-1}) \otimes \chi_F \pi_{\mathrm{Koop}}(\theta)\chi_F, \sigma
\in G.
\end{equation*}

\noindent Then we have that 
\begin{equation*}
\widetilde{\varepsilon}(S(\Gamma\sigma\Gamma)) = \sum \chi_F \pi_{\mathrm{%
Koop}}(\theta)\chi_F, \sigma \in G.
\end{equation*}
This is exactly the classical Hecke operator associated to the double coset $%
\Gamma\sigma\Gamma$ that is  acting on the  space $L^2(F,\nu)$ associated to the fundamental
domain.
\end{remark}

\vskip6pt

\begin{remark}
In the previous more general setting of Remark \ref{exp}, we identify the $\Gamma$- wandering
subspace $L$ of $H$, with a subspace of $\Gamma$ invariant vectors in $H$
(here by $\Gamma$ - invariant vectors we understand densely defined, $\Gamma$
- invariant functionals on $H$).

Thus to every vector $l \in L$, we associate the $\Gamma$ - invariant vector 
$\mathop{\sum}\limits_{\gamma \in \Gamma}\pi(\gamma)l$. Then clearly 
\begin{equation*}
\widetilde{\varepsilon}(S(\Gamma\sigma\Gamma))l= \widetilde{\varepsilon}%
(S(\Gamma\sigma\Gamma))(P_L\mathop{\sum}\limits_{\gamma \in
\Gamma}\pi(\gamma)l) = \mathop{\sum}\limits_{\theta \in \Gamma\sigma\Gamma}
P_L\pi(\theta)l, l\in L,
\end{equation*}

\noindent This is the further equal to

\vskip6pt

\centerline{$P_L(\sum \pi(\sigma)\pi(s_i)(\mathop{\sum}\limits_{\gamma \in
\Gamma}\pi(\gamma)l))$},

\noindent and this is exactly the action of the classical Hecke operator
associated to the double coset $\Gamma\sigma\Gamma$ on $\sum \pi(\gamma)l$
(which is $\sum \pi(\sigma s_i)\mathop{\sum}\limits_{\gamma \in
\Gamma}\pi(\gamma)l = \mathop{\sum}\limits_{\gamma \in \Gamma}\pi(\gamma)(%
\mathop{\sum}\limits_{v_j}\pi(\sigma v_j)l)$).
Here $s_i, v_j$; are cosets representatives so that $\Gamma \sigma\Gamma =
\cup s_i\sigma\Gamma = \cup \Gamma\sigma v_j$). In particular we obtain an
extension from classical Hecke operators, which are a representation of ${%
\mathcal{H}}_0$ to a representation of ${\mathcal{K}}_o^{\ast}{\mathcal{K}}%
_0 $.
\end{remark}

One example, in which $\widetilde{\varepsilon} \circ S$ is well defined, is
the case of a tensor product of two representations of the type of those  considered above.

\vskip6pt

\begin{proposition}
Let $\pi_i, G, \Gamma, L_i, H_i, i=1,2$ be as above and let $\pi : G \to \mathcal U(H_1
\otimes \overline{H}_2)$ be the diagonal representation. We identify $%
H_1\otimes \overline{H}_2$ with ${\mathcal{C}}_2(H_1, H_2)$, the Schatten
von Neumann class of operators from $H_1$ into $H_2$. In this
identification, the representation $\pi$ is the adjoint representation,
mapping $X \in {\mathcal{C}}_2(H_1, H_2)$ into $\pi_2(\gamma) X
\pi_1(\gamma^{-1})$.

For simplicity we assume that $\Gamma$ has infinite conjugacy non-trivial
classes, so that its associated von Neumann algebra has a unique trace $\tau$%
.

Then the $\pi(\Gamma)$ - invariant vectors in ${\mathcal{C}}_2(H_1, H_2)$
are identified with the $L^2$ - space of $\Gamma$ - intertwiners of the
representations $\pi_1,\pi_2$. We denote the space of $\Gamma$-intertwiners
by $\mathrm{Int}_{\Gamma}(\pi_1, \pi_2)$. If $H_1 = H_2, \pi_1 = \pi_2$ then
this space is simply the commutant algebra $\pi_1(\Gamma)^{\prime}\subseteq
B(H_1)$. The $L^2$ - space associated to the space of intertwiners is canonically determined by the trace
on the von Neumann factor $\pi_1(\Gamma)^{\prime}$.

In general the $L^2$ - space of $\mathrm{Int}_{\Gamma}(\pi_1, \pi_2)$,
denoted  by $L^2(\mathrm{Int}_{\Gamma}(\pi_1, \pi_2), \tau)$, is
obtained by first fixing one of the two spaces, which ever has greater equal
Murray von Neumann dimension. Assume, for example,that this space is $H_1$.
Then the $L^2$ - norm of $X$ is $\tau_{\pi_1(\Gamma)^{\prime}}(X^{%
\ast}X)^{1/2}$.

The classical Hecke operator on $\Gamma$ invariant vectors associated to the
representation $\pi$ is then the extension to the $L^2$ space of
intertwiners, of the operator $\Psi(\Gamma\sigma\Gamma)$ defined by the
following construction:  for $\sigma \in G$, assuming that the disjoint decomposition
of the coset $\Gamma\sigma\Gamma$ is $\cup_i s_i\sigma\Gamma$, we let 
\begin{equation*}
\Psi(\Gamma\sigma\Gamma)(X) = \sum_i
\pi_2(s_i\sigma)X\pi_1(\sigma^{-1}s_i^{-1}), X\in \mathrm{Int}%
_{\Gamma}(\pi_1, \pi_2).
\end{equation*}

\noindent This is obviously a bounded operator.
We identify $\mathrm{Int}_{\Gamma}(\pi_1, \pi_2)$ with $R(\Gamma) \otimes
B(L_1, L_2)$ and we identify the associated Hilbert space $L^2(\mathrm{Int}%
_{\Gamma}(\pi_1, \pi_2),\tau)$ with $\l ^2(\Gamma)\otimes B(L_1, L_2)$. Let $%
S_i : {\mathcal{H}}_0 \to R(G) \otimes B(L_i)$, for $i=1,2$, be the
representation of the Hecke algebra associated with the representations $%
\pi_1, \pi_2$ constructed in Proposition \ref{rep}. Let $S$ be the
representation of the Hecke algebra ${\mathcal{H}}_0$, given by Proposition %
\ref{rep}, associated to the representation $\pi$.

Then, for every double coset $\Gamma\sigma\Gamma$, the operator $\widetilde{%
\varepsilon}(S(\Gamma\sigma\Gamma))$ is equal to $\Psi(\Gamma\sigma\Gamma)$. In
particular, the construction in Proposition \ref{rep} for the representation 
$\pi$, yields the classical Hecke operators on $\mathrm{Int}_{\Gamma}(H_1,
H_2)$. Moreover, we have the following formula for $\Psi(\Gamma\sigma\Gamma)$%
: 
\begin{equation*}
\Psi(\Gamma\sigma\Gamma) (X)= E_{R(\Gamma) \otimes B(L_1, L_2)}^{R(G)
\otimes B(L_1, L_2)} (S_2(\Gamma\sigma\Gamma)XS_1(\Gamma\sigma\Gamma)),
\end{equation*}
for all 
\begin{equation*}
X \in \mathrm{Int}_{\Gamma}(\pi_1, \pi_2) \cong R(\Gamma) \otimes B(L_1,
L_2).
\end{equation*}
Here $E_{R(\Gamma) \otimes B(L_1, L_2)}^{R(G) \otimes B(L_1, L_2)}$ is the
canonical conditional expectation.
\end{proposition}

\vskip6pt

\begin{proof} We identify $H_i$ with $l^2(\Gamma) \otimes L_i$.
Then we have that $$H_1 \otimes \overline{H}_2 = l^2(\Gamma) \otimes \overline{l^2(\Gamma)} \otimes L_1 \otimes \overline{L}_2,$$ and hence, a possible $\Gamma$ wandering subspace for $H_1 \otimes \overline{H}_2$ is $l^2(\Gamma) \otimes L_1 \otimes \overline{L}_2$. This space is identified with $L^2({\rm Int}_{\Gamma}(H_1, H_2),\tau)$ as follows:

Fix $l_i, m_i$ vectors in $L_i, i = 1,2$. Fix $\gamma_a, \gamma_b$ in $\Gamma$. Let $e$ be the identity element of the group $\Gamma$. Consider the following 1-dimensional operators: $$A_0 = \langle e \otimes l_1, \cdot \rangle(\gamma_a \otimes l_2), B_0 = \langle e \otimes m_1, \cdot \rangle (\gamma_b \otimes m_2).$$
 Let $$A = \sum \pi_1(\gamma)A_0 \pi_2(\gamma)^{-1}, B = \sum \pi_1(\gamma)B_0 \pi_2(\gamma)^{-1}.$$
Then $$\tau(AB^{\ast}) = {\rm Tr}(AB_0).$$

We obtain that $$\Psi(\Gamma\sigma\Gamma)(A) = \mathop{\sum}\limits_{\theta \in \Gamma\sigma\Gamma}\pi_2(\theta)A_0\pi_1(\theta)^{-1}.$$
The matrix coefficients are
$$
\mathop{\sum}\limits_{\theta \in \Gamma\sigma\Gamma} {\rm Tr}(\pi_2(\theta)\left[ \langle e \otimes l_1, \cdot \rangle (\gamma_a \otimes l_2) \right]\pi_1(\theta)^{-1}(\langle e \otimes m_1, \cdot \rangle \gamma_b \otimes m_2)^{\ast})=
$$
$$
= \mathop{\sum}\limits_{\theta \in \Gamma\sigma\Gamma} {\rm Tr}(\left[ \langle \pi_1(\theta)^{-1}(e \otimes l_1), \cdot \rangle \pi_2(\theta)(\gamma_a \otimes l_2) \right]\left[ \langle e \otimes m_1, \cdot \rangle (\gamma_b \otimes m_2)\right]^{\ast})=
$$
$$
= \mathop{\sum}\limits_{\theta \in \Gamma\sigma\Gamma} \langle \pi_1(\theta)^{-1}(e \otimes l_1), \gamma_a \otimes l_2 \rangle \langle \pi_2(\theta)(e \otimes m_1)(\gamma_b \otimes m_2) \rangle.
$$

But this are exactly the matrix coefficients of $\widetilde{\varepsilon}(S)(\Gamma\sigma\Gamma)$ where $S$ is the representation of ${\mathcal{H}}_0$ associated to the representation $\pi_1 \otimes \overline{\pi}_2$ on the vectors in the $\Gamma$ - wandering subspace $$l^2(\Gamma) \otimes L_1 \otimes \overline{L}_2 \cong \Bbb C e \otimes \overline{l_2(\Gamma)} \otimes L_1 \otimes \overline{L}_2,$$ corresponding to the vectors $e \otimes \gamma_a \otimes l_1 \otimes l_2$ and $e \otimes \gamma_b \otimes m_1 \otimes \overline{m}_2$.

To verify the formula for the conditional expectation, note that it is sufficient to check that we get the same values for the matrix entries, when evaluated at elements $A, B$ as above.

The trace on $R(\Gamma) \otimes B(L_1)$ is $\tau_{R(\Gamma)} \otimes {\rm tr}_{B(L_1)}$.
Thus we have to compute for $A, B \in R(\Gamma) \otimes B(L_1, L_2)$,
$$
\tau_{R(\Gamma) \otimes B(L_1)}(E_{R(\Gamma) \otimes B(L_1, l_2)}^{R(G) \otimes B(L_1, L)}(S_2(\Gamma\sigma\Gamma)AS_1(\Gamma\sigma\Gamma)B^{\ast}))=
$$
$$
= \tau_{R(\Gamma) \otimes B(L_1)} (S_2(\Gamma\sigma\Gamma)AS_1(\Gamma\sigma\Gamma)B^{\ast})=
$$
$$
= \tau ((\mathop{\sum}\limits_{\theta_1\in \Gamma\sigma\Gamma}\rho(\theta_1^{-1}) \otimes P_L\pi_2(\theta_1)P_L)(R_{\gamma_a} \otimes \langle l_1, \cdot \rangle l_2)
$$
$$
(\sum\limits_{\theta_2\in \Gamma\sigma\Gamma} \rho(\theta_2^{-1}) \otimes P_L\pi_1(\theta_2)P_L)R_{\gamma_b^{-1}} \otimes \langle m_1, \cdot \rangle m_2) =
$$
$$
= \mathop{\sum}\limits_{\mathop{\theta_1, \theta_2 \in \Gamma\sigma\Gamma}\limits_{\theta_1\gamma_a\theta_2\gamma_b^{-1} = e}} {\rm Tr}((P_L\pi_2(\theta_1)P_L)(\langle l_1, \cdot \rangle l_2)P_L\pi_1(\theta_2)P_L \langle m_1, \cdot \rangle m_2) =
$$

\noindent (as $\theta_2 = \gamma_b \theta_1^{-1}\gamma_a$)
$$
= \mathop{\sum}\limits_{\theta \in \Gamma\sigma\Gamma} \langle P_L \pi_2(\theta)P_L l_2, m_2 \rangle \overline{\langle P_L \pi_1(\gamma_a\theta_1^{-1}\gamma_b)P_L l_1, m_1 \rangle} =
$$
$$
= \mathop{\sum}\limits_{\theta \in \Gamma\sigma\Gamma} \langle \pi_2(\theta)l_2, m_2 \rangle \overline{\langle \pi_1(\gamma_a^{-1}\theta_1\gamma_b^{-1})l_1, m_1 \rangle}
$$

\noindent which is exactly the previous formula.
\end{proof}

\vskip6pt

\section{The case on non-integer Murray and von Neumann dimension}

In this section we extend the results in the previous section to the case of
non-integer Murray von Neumann dimension. We assume that there exists a
representation $(\pi_0, \Gamma, G, H_0)$ as in the previous section, but we
do not assume that there exists a splitting $H_0 = l^2(\Gamma) \otimes L_0$.

Instead we assume that there exists a representation $(\pi ,G)$ on a larger
Hilbert space $H$, containing $H_0$, which has a splitting $H=l^{2}(\Gamma )\otimes L$, with $%
\pi |_{\Gamma }=\lambda _{\Gamma }\otimes \mathrm{Id}_{L}$, and such that if 
$p:H\rightarrow H_{0}$ is the canonical orthogonal projection, then $\pi (G)$
commutes with $p$ and $\pi _{0}(g)=\pi (g)|_{H_{0}}$. Also we assume $pP_{L}$
is trace class. Denote by $S:{\mathcal{H}}_{0}\rightarrow \widehat{%
R(G)\otimes B(L)}$ the representation of ${\mathcal{H}}_{0}$ constructed in
the previous section.

We will prove that $p$ also commutes with $S$, and that the correspondence $$\Gamma\sigma\Gamma\rightarrow
pS(\Gamma\sigma\Gamma)p,\ \  \sigma \in G$$ extends to   a representation of ${\mathcal{H}}_0$ into $%
p(R(\Gamma) \otimes B(L))p \cong \pi_0(\Gamma)^{\prime}$, with the required
properties. Note that this 
representation of ${\mathcal{K}}^{\ast}$ takes values into $p(\mathcal R(\Gamma)
\otimes B(L))$.

\vskip6pt

\begin{lemma} Let $\pi, G, H, L$ as above, and let the subspace $H_0 = pH$, where $p$ is a projection such that $[p, \pi(g)] = 0$, for all $g$ in $G$. Let $\pi_0(g) = \pi(g)|_{H_0}, g \in G$. We assume also that $pP_L$ is trace class $P_L$.

Then $p$ commutes with $S(\Gamma\sigma\Gamma)$, for $\sigma$ in $G$.
 The explicit expression for $p$ is
$$
\sum \rho(\gamma) \otimes P_L \pi_0(\gamma^{-1})P_L = E_{(\pi_0(\Gamma))'}(pP_L) = E_{\pi_0(\Gamma)'}(pP_L p).
$$

\noindent In this case, since $pP_L$ is trace class, the above expression belongs to $\mathcal R(\Gamma) \otimes B(L)$.

Then
$$
pS(\Gamma\sigma) = \mathop{\sum}\limits_{\theta \in \Gamma\sigma} \rho(\theta^{-1}) \otimes P_L \pi_0(\theta)P_L \in p(R(G) \otimes B(L)),
$$
and
$$
pS(\Gamma\sigma\Gamma) = S(\Gamma\sigma\Gamma)p = \mathop{\sum}\limits_{\theta \in \Gamma\sigma\Gamma} \rho(\theta^{-1}) \otimes P_L \pi_0(\theta)P_L.
$$

\noindent Hence $S_p(\Gamma\sigma\Gamma) = pS(\Gamma\sigma\Gamma)p = pS(\Gamma\sigma\Gamma)$ determines a representation $S_p$ of $\mathcal H_0$, while $\Gamma\sigma \to pS(\Gamma\sigma)$ determines a representation o $\mathcal K^{\ast}\mathcal K$ with $$S(\sigma_1\Gamma)pS(\Gamma\sigma_2) = \mathop{\sum}\limits_{\theta \in \sigma_1\Gamma\sigma_2} \rho(\theta^{-1}) \otimes P_L \pi_0(\theta)P_L.$$
\end{lemma}

\vskip6pt

\begin{proof} The expression for $p$ follows from the statement of Proposition \ref{rep} in Section 2.

Then all the statements are a consequence of the formula
$$
\mathop{\sum}\limits_{\gamma} P_L\pi_0(\gamma)P_L \pi(\gamma^{-1}\theta)P_L = P_L\pi_0(\theta)P_L.
$$

Indeed, for example, when doing
$$
pS(\Gamma\sigma) = (\mathop{\sum}\limits_{\gamma \in \Gamma} \rho(\gamma^{-1}) \otimes P_L\pi_0(\gamma)P_L)(\mathop{\sum}\limits_{\theta \in \Gamma\sigma}\rho(\theta^{-1}) \otimes P_L\pi(\theta)P_L) =
$$
$$
= \mathop{\sum}\limits_{\theta_1 \in \Gamma\sigma} \rho(\theta_1^{-1}) \otimes \mathop{\sum}\limits_{\mathop{\theta \in \Gamma\sigma}\limits_{\gamma\theta = \theta_1}} P_L\pi_0(\gamma)P_L\pi(\theta)P_L =
$$
$$
= \mathop{\sum}\limits_{\theta_1 \in \Gamma\sigma} \rho(\theta_1^{-1}) \otimes \mathop{\sum}\limits_{\gamma \in \Gamma} P_L \pi_0(\gamma)P_L\pi(\gamma^{-1}\theta_1)P_L =
$$
$$
= \mathop{\sum}\limits_{\theta_1 \in \Gamma\sigma} \rho(\theta_1^{-1}) \otimes P_L\pi_0(\theta_1)P_L.
$$

To check that $S_p(\Gamma\sigma\Gamma) = pS(\Gamma\sigma\Gamma)p$ is indeed a representation of $\mathcal H$ we have to verify that it implements the Hecke operators $\Psi_{\Gamma\sigma\Gamma}$ on $\pi_0(\Gamma)'$.

Here $\Psi_{[\Gamma\sigma\Gamma]}$ is described as the transformation mapping\break $X' \to\sum \pi(s_i\sigma)X'\pi(\sigma s_i)^{-1}$, for $X'$ in $\pi_0(\Gamma)'$, where $\Gamma\sigma\Gamma=\cup s_i\sigma\Gamma$ is  a double coset. As we mention above, extending  by continuity $\Psi_{[\Gamma\sigma\Gamma]}$ to the $L^2$ - space we obtain  the composition  of $\widetilde{\varepsilon}$ with the Hecke operators for $H \otimes H$.

Here $\pi_0(\Gamma)'$ is isomorphic (as $p$ belongs to $R(\Gamma) \otimes B(L)$) to $$p\pi_0(\Gamma)' p= p(R(\Gamma) \otimes B(L))p.$$
\end{proof}

We have

\vskip6pt

\begin{theorem}\label{expectation}
\label{t} Let $\pi_0, \pi, p$ be as above. Let $S_p : {\mathcal{H}}_0 \to
p(R(G) \otimes B(L))p$ be the representation constructed in the previous
proposition. \ Then we have 
\begin{equation*}
E_{p(R(\Gamma) \otimes B(L))p}^{p(R(G) \otimes B(L))p}
(S_p(\Gamma\sigma\Gamma)XS_p(\Gamma\sigma\Gamma)) =
%\end{equation*}
%\begin{equation*}
 \Psi_{[\Gamma\sigma\Gamma]}(X), \end{equation*}
\; for \; all \; $X$ \; in \; $p(R(\Gamma)
\otimes B(L))p$, $\sigma \in G$.

\end{theorem}

\vskip6pt

\begin{proof} Let $P_L$ be the orthogonal projection onto $L$. Let $\sigma \in G$ and assume that $\Gamma\sigma\Gamma = \cup \Gamma \sigma s_i$, a disjoint union.
Then $$\Psi_{[\Gamma\sigma\Gamma]}(X) = \sum_{i}\pi_0(\sigma s_i)^{-1}X\pi_0(\sigma s_i).$$

An element $X'$ in $\pi_0(\Gamma)'$ has the following representation in $R(G) \otimes B(L)$
$$
\sum \rho(\gamma^{-1}) \otimes P_L\pi_0(\gamma)X'P_L.
$$

 We  compute first the product
$$
\mathop{\sum}\limits_{\mathop{\theta_1, \theta_2 \in \Gamma\sigma\Gamma}\limits_{\gamma \in \Gamma}} \left[\rho(\theta_1^{-1}) \otimes P_L\pi_0(\theta_1)P_L\right]\left[\rho(\gamma^{-1}) \otimes P_L\pi_0(\gamma)X'P_L\right]\left[\rho(\theta_2^{-1})\otimes P_L\pi_0(\theta_2))P_L\right].
$$

But this is equal to
$$
\mathop{\sum}\limits_{\theta \in (\Gamma\sigma\Gamma)^2} \rho(\theta^{-1}) \otimes \mathop{\sum}\limits_{\mathop{\theta_1, \theta_2 \in \Gamma\sigma\Gamma,\; \gamma \in \Gamma}\limits_{\theta_1\gamma\theta_2 = \theta}} P_L\pi_0(\theta_1)P_L\pi_0(\gamma)X'P_L\pi_0(\theta_2)P_L.
$$

In the second sum we necessary have $\theta_1 = \theta\theta_2^{-1}\gamma^{-1}$. Hence the above sum becomes
$$
\mathop{\sum}\limits_{\theta \in (\Gamma\sigma\Gamma)^2} \rho(\theta^{-1}) \otimes \mathop{\sum}\limits_{\theta_2 \in \Gamma\sigma\Gamma} P_L\pi_0(\theta\theta_2^{-1})X'P_L\pi_0(\theta_2)P_L. \eqno(\ast)
$$

We use the decomposition $\Gamma\sigma\Gamma = \cup \Gamma\sigma s_i$.

Then the sum of terms for a fixed $\sigma s_i$ corresponds to
$$
\mathop{\sum}\limits_{\gamma \in \Gamma} P_L\pi_0(\theta)\pi_0(\sigma s_i)^{-1}\pi_0(\gamma)X'P_L\pi_0(\gamma)\pi_0(\sigma s_i)
$$

\noindent which since $[X', \pi_0(\gamma)]$, and again since $\mathop{\sum}\limits_{\gamma} \pi_0(\gamma)P_L\pi_0(\gamma)^{-1} = p$, the above sum is equal to  
$$
P_L\pi_0(\theta)\pi_0(\sigma S_i)^{-1}X'\pi_0(\sigma S_i)P_L = P_L\pi_0(\theta)\Psi_{[\Gamma\sigma\Gamma]}(X')P_L.
$$

Thus
$$
S_p(\Gamma\sigma\Gamma)X'S_p(\Gamma\sigma\Gamma) = \mathop{\sum}\limits_{\theta \in (\Gamma\sigma\Gamma)^2} \rho(\theta^{-1}) \otimes P_L\pi_0(\theta)\Psi(X')P_L.
$$

Note that by $(\Gamma\sigma\Gamma)^2$ we simply understand the set of double cosets contained in the set  $\Gamma\sigma\Gamma\sigma\Gamma$, with no multiplicities (since in the sum after $\theta$, every $\theta$ appears only once).

Now an easy argument of the same type as above shows that for $X', Y' \in \pi_0(\Gamma)'$
$$
A\cdot B = (\mathop{\sum}\limits_{\theta \in \sigma\Gamma} \rho(\theta^{-1}) \otimes P_L\pi_0(\theta)Y'P	_L)(\mathop{\sum}\limits_{\gamma \in \Gamma} \rho(\gamma^{-1}) \otimes P_L\pi_0(\gamma)X'P_L) =
$$
$$
= \mathop{\sum}\limits_{\theta_1 \in \sigma\Gamma} \rho(\theta_1^{-1}) \otimes P_L\pi_0(\theta_1)Y'X'P_L
$$

\noindent thus the product  $AB$ has zero trace, and thus the two sums corresponding to the two factors in the product $AB$ above are orthogonal.

Hence
$$
E_{p(R(\Gamma) \otimes B(L))p}^{p(R(G) \otimes B(L))p} (\mathop{\sum}\limits_{\theta \in (\Gamma\sigma\Gamma)^2}\rho(\theta^{-1}) \otimes P_L\pi_0(\theta)\Psi(X')P_L) =
$$
$$
= \mathop{\sum}\limits_{\gamma \in \Gamma} \rho(\gamma^{-1}) \otimes P_L\pi_0(\gamma)\Psi(X')P_L
$$

\noindent which is exactly $\Psi(X')$.
\end{proof}

\vskip6pt

We observe that the representation of ${\mathcal{H}}_{0}$ obtained by this
method, in the case when $\mathrm{dim}_{\pi _{0}(\Gamma )}H_{0}$ an integer, is
the same as the one obtained in the previous section.

We prove the following lemma of abstract character. It will be used to prove
that the representation $S_p$ is the same as the representation we
constructed in ([Ra]). This technique can also be used to evaluate traces of
Hecke operators as we explain bellow.

\vskip6pt

\begin{lemma}\label{avn} Let $M$ be a type $II_1$ factor with trace
$\tau$ and let $\widetilde{\mathcal{M}} =M \otimes B(H)$ be the associated type $II_{\infty}$ factor with trace $T = \tau \otimes {\rm Tr}$. Let  $N$ be the subfactor $M \otimes I$ of $\widetilde{\mathcal{M}}$,  and let  $p$ be a finite projection in $\widetilde{\mathcal{M}}$.

Assume that $E_N(p)$ is invertible (which, as we will observe, automatically implies $T(p) = 1$).

Then the map
$$
\Phi : p\widetilde{\mathcal{M}}p \to M
$$

\noindent defined by
$$
\Phi(pmp) = E_N(p)^{-1/2}E_N(pmp)E_N(p)^{-1/2}, \; pmp \in pMp
$$

\noindent is an algebra isomorphism.
\end{lemma}

\vskip6pt

\begin{proof} We consider the Jones basic construction algebra $\widetilde{\mathcal{M}}_1 = \langle \widetilde{\mathcal{M}}, e_N \rangle$ associated to $N \subseteq M \otimes 1$, with trace $\widetilde{T}$.

Then $e_N$ commutes with $N$ and the map $x \in N \to e_Nx$ is as isomorphism.

Then, by using this isomorphism, the statement is equivalent to prove the fact that the map $\Phi e_N$ on $pMp$, associating:
$$
pmp \to (e_Npe_N)^{-1/2}e_Np(pmp)e_Np(e_Npe_N)^{-1/2},
$$

\noindent is an isomorphism.

Let $v = (e_Npe_N)^{-1/2}_Np$ be the partial isometry from the polar decomposition of $e_Np$. Since  $e_Npe_N$ is invertible, $v$ is an isometry from $p$ onto $e_N$. (Since  $\widetilde{T}(e_N) = 1$ we  necessary have that $\widetilde{T}(p) = T(p) = 1$).

Hence $\Phi e_N = {\rm Ad}v : p\widetilde{M}p \to e_N\widetilde{M}e_N = Ne_N$ and hence $\Phi$ is an algebra isomorphism.
\end{proof}

\vskip6pt

\begin{remark}
\label{convex} If $E_N(p)$ is not an invertible element, then we assume that there
exist positive scalars $\lambda_i$, with $\sum\lambda_i = 1$, and unitaries $%
u_i \in M$ identified with $u_i \otimes 1$ such that $\sum%
\lambda_iu_iE_N(p)u_i^{\ast}$ is invertible. We then the replace $\widetilde{%
M}$ by $\mathop{M}\limits^{\approx} =\mathop{\oplus}\limits_{n}^{i =1} 
\widetilde{M}$ where each component has weight $\lambda_i^{1/2}$ and
consider the embedding $M \otimes 1 \subseteq \mathop{M}\limits^{\approx}$%
, through the maps $m \otimes 1 \to \oplus\; u_imu_i^{\ast} \otimes 1$.

Then the conditional expectation of $\widetilde{p} = \mathop{\oplus}%
\limits_{n}^{i =1} p \in \mathop{M}\limits^{\approx}$ is $\sum
\lambda_iu_i^{\ast}E_N(p)u_i$ and hence is invertible. Hence we can still
apply the previous lemma to this representation.
\end{remark}

\ 

\

We consider first the one dimensional (in the sense of the Murray and von
Neumann dimension theory). In this case $\mathrm{dim}_{\pi(\Gamma)}H_0 = 1$.
If $\eta$ is a cyclic trace vector, then we can choose a subspace $L_0 = 
\mathbb{C}\eta$ and hence the projection $P_{L_0}$ (introduced in the proof of Theorem \ref{expectation} will be $\langle \eta,
\cdot \rangle \eta$ and hence the map $S : {\mathcal{H}} \to R(G) \otimes 
\mathbb{C}\langle \eta, \cdot \rangle \eta$ is 
\begin{equation*}
S(\Gamma\sigma\Gamma) = \mathop{\sum}\limits_{\theta \in \Gamma\sigma\Gamma}
\rho(\theta)\otimes P_{L_0}\pi_0(\theta)P_{L_0} =
\end{equation*}
\begin{equation*}
= \mathop{\sum}\limits_{\theta \in \Gamma\sigma\Gamma} \rho(\theta) \otimes
\langle \pi(\theta)\eta, \eta \rangle P_{L_0}.
\end{equation*}
Identifying $\mathbb{C }P_{L_0}$ with $\mathbb{C}$, we get exactly the
representation in [Ra].

\vskip6pt

\begin{proposition}
\label{ep} We use  the notations from the previous proposition.
Let $N=R(\Gamma )\otimes 1$ and let $\zeta =E_{N}(p)^{-1/2}$.
 Since $$p = \sum
\rho(\gamma^{-1}) \otimes P_{L_0}\pi_0(\gamma)P_L,$$
 we obtain  the equalities:
\begin{equation*}
E_{R(G) \otimes 1}(p) = \sum\limits_{\gamma \in \Gamma}\rho(\gamma^{-1}) \mathrm{Tr}(P_L\pi_0(\gamma)),
\end{equation*}
\begin{equation*}
E_{R(G) \otimes 1}(S_p(\Gamma\sigma\Gamma)) = \mathop{\sum}\limits_{\theta
\in \Gamma\sigma\Gamma} \rho(\theta) \mathrm{Tr}(P_L\pi_0(\theta)).
\end{equation*}

 If $E_{N}(p)$
is invertible, then the map on ${\mathcal{H}}$, obtained by mapping $[\Gamma \sigma
\Gamma ]$ into $S_{p}(\Gamma \sigma \Gamma )$, constructed in this section,
is unitarily equivalent, by Lemma \ref{avn}, to the representation determined by the correspondence:
\begin{equation*}
\lbrack \Gamma \sigma \Gamma ]\rightarrow \zeta ^{-1/2}\mathop{\sum}\limits%
_{\theta \in \Gamma \sigma \Gamma }\rho (\theta )\mathrm{Tr}(P_{L}\pi
_{0}(\theta ))\zeta ^{-1/2}, \ \  \sigma \in G.\eqno(\ast )
\end{equation*}

\noindent If $E_{R(\Gamma) \otimes 1}(p)$ is not invertible, by [BH] we may find $g_1,
\ldots, g_n \in \Gamma$ and positive scalars $\lambda_i, \; \sum\lambda_i =
1 $, such that 
\begin{equation*}
\mathop{\sum}\limits_{i} \lambda_i(\rho(g_i) \otimes 1) E_{R(\Gamma) \otimes
1}(p)(\rho(g_i)^{-1} \otimes 1)
\end{equation*}

\noindent is invertible, and   use the Remark \ref{convex}, to
obtain a representation of the algebra ${\mathcal{H}_0}$.
\end{proposition}

Proof. This follows from the previous two results.

\vskip6pt

\begin{remark}
The representation $(\ast)$ in the previous statement is exactly the map corresponding to $%
[\Gamma\sigma\Gamma] \to \mathop{\sum}\limits_{\theta \in
\Gamma\sigma\Gamma} \rho(\theta)\langle \pi(\theta)\eta, \eta \rangle$ where 
$\eta$ is a cyclic trace vector. 

In the general case, assume that,  $$P_LpP_L = \sum \langle \zeta_i, \cdot
\rangle\zeta_i,$$ where $\zeta_i$ is an orthogonal basis. Let $\widetilde{%
\zeta_i} = \lambda_i^{1/2}\zeta_i$, which is an orthonronal system. Then $$%
\widetilde{\eta} = \oplus \; \widetilde{\zeta}_i \in H_0^{\aleph_0} = 
\widetilde{H}_0$$ is an unit vector. Let $\widetilde{\pi} = \oplus \; \pi$
acting diagonally on $\widetilde{H}_0$.

Then $$\langle \widetilde{\pi}(\theta)\widetilde{\eta}, \widetilde{\eta}
\rangle = \mathrm{Tr}(P_L\pi_0(\theta)) = \mathrm{Tr}(pP_Lp\pi(\theta)).$$
\end{remark}

\vskip6pt

\begin{proof}

In this case $\widetilde{\eta}$ is not necessary a $\Gamma$ - wandering vector, but $\widetilde{\eta}_0 = \pi(\zeta)^{-1/2}\widetilde{\eta}$ is a $\Gamma$ - wandering vector. Because $\widetilde{\pi}$ is a multiple of $\pi$, it follows that $Sp\overline{\widetilde{\pi}(\gamma)\widetilde{\eta}}$ is invariant under $\widetilde{\pi}(g)$, $g \in G$. Thus, the coefficients of the representation of $\mathcal H_0 \to B(\widetilde{H}_0)$ associated to $\tilde \pi$, as in Section 2, are
$$
[\Gamma\sigma\Gamma] \to  \mathop{\sum}\limits_{\theta\in \Gamma\sigma\Gamma} \rho(\theta)\langle \widetilde{\pi}(\widetilde{\eta}_0), \widetilde{\eta}_0 \rangle \in R(G).
$$

But this sum is exactly $$\zeta^{-1/2}\mathop{\sum}\limits_{\theta\in \Gamma\sigma\Gamma} \rho(\theta) {\rm Tr}(P_L\pi_0(\theta))\zeta^{-1/2} \in R(G), \ \ \sigma \in G.$$

\end{proof}

\vskip6pt

In the case $n > 1$, $\mathrm{dim}_{\pi(\Gamma)}H_0 = n$, we consider only
the case $H = L^2(X, \mu)$, $L = L^2(F, \mu)$.

Here $(X, \mu)$ is an infinite measure space. Assume that $G$ acts on $X$ by
measure preserving transformations and $F$ is a fundamental domain for $%
\Gamma$ in $X$.

We divide $F$ into $n$ equal parts. We assume that $v_{i1}$, the isometry
from the polar decomposition of $\chi_{F_i}p\chi_{F_1}$ is invertible.

Then $\pi(\Gamma)^{\prime}$ is isomorphic to $R(\Gamma) \otimes M_n(\mathbb{C%
}) \otimes L^2(F_1)$ and a given element $X^{\prime}$ in $%
\pi(\Gamma)^{\prime}$ is represented as 
\begin{equation*}
\mathop{\sum}\limits_{\gamma\in\Gamma} \rho(\gamma^{-1}) \otimes
(\chi_{F_1}v_{1i}\pi_0(\gamma)X^{\prime}v_{j1}\chi_{F_1})_{ij}
\end{equation*}

\noindent while 
\begin{equation*}
S_p(\Gamma\sigma\Gamma) = \mathop{\sum}\limits_{\theta \in
\Gamma\sigma\Gamma} \rho(\theta^{-1}) \otimes
(\chi_{F_1}v_{1i}\pi_0(\theta)v_{j1}\chi_{F_1})_{ij}, \; \sigma \in G.
\end{equation*}

Applying the previous lemma, and assuming that $E_{R(\Gamma) \otimes M_n(%
\mathbb{C})}(p)$ is invertible, we get that thus representation of ${%
\mathcal{H}}$ is unitary equivalent to the representation $S_0$ mapping $%
[\Gamma\sigma\Gamma]$ into $\mathop{\sum}\limits_{\theta \in
\Gamma\sigma\Gamma} \rho(\theta) \otimes \mathrm{Tr}(\chi_Fv_{1i}\pi_0(%
\theta)v_{j1}\chi_{F_1})$, for $\sigma \in G$.

Hence $\varepsilon(S_0(\Gamma\sigma\Gamma))$ is $\mathop{\sum}%
\limits_{\theta \in \Gamma\sigma\Gamma} (\mathrm{Tr}(v_{j1}v_{1i}\pi_0(%
\theta)))_{ij}$.

Thus the trace of the Hecke operators is  the trace of the matrix
\begin{equation*}
(\mathrm{Tr}((v_{j1}v_{1j})\sum \pi_0(\theta)))_{ij}.
\end{equation*}

This is similar to the formula in [Za]. We remark that for $n = 13$, $\pi =
\pi_{13}$, the 13-th element in the discrete series of $PSL_2(\mathbb{R})$
on $H_{13}$, $G = PGL_2(\mathbb{R}[\frac{1}{p}])$, $\Gamma = PSL_2(\mathbb{Z}%
)$, the trace of the Hecke operator on "$\Gamma$ - invariant" vectors in $%
H_{13}$ results to be 
\begin{equation*}
\mathop{\sum}\limits_{\theta \in \Gamma\sigma\Gamma} \mathrm{Tr}%
(\pi_0(\theta)\chi_F) = \mathop{\sum}\limits_{\theta \in \Gamma\sigma\Gamma} %
\mathop{\int}\limits_{F} (\dfrac{1}{1 - \bar{z}\theta(z)})^t d\Gamma_t(z).
\end{equation*}

\noindent (see also next section for another interpretation).

\vskip6pt

\section{The relation to classical Hecke operators}

In this section we determine the relation between the action of the Hecke
operators on $\Gamma $ invariant vectors and the representation of the Hecke
operators constructed in the previous section. This is done in the absence
of a splitting space for the given representation (corresponding to Hilbert space tensorial splitting of the type $%
H_{0}=l^{2}(\Gamma )\otimes L_{0}$). Instead we will perform this construction, in the presence of a splitting for
a larger representation of $G$, that contains the given representation as a
super-representation. We prove that the representations and traces of Hecke
operators are obtained by applying the (unbounded) character $\varepsilon
:R(G)\rightarrow \mathbb{C}$ (the sum of coefficients) to the
representations of the Hecke algebra constructed in the previous two sections.

Let $H_0 \subseteq H$ be Hilbert space, $\pi : G \to {\mathcal{U}}(H)$ a
representation invariating $H_0$ and let $\pi_0 = \pi|_{H_0}$.

Denote by $p$ the projection onto $H_0$. We assume that $L$ is a $\Gamma$ -
wandering, generating subspace of $H$, thus $H \cong l^2(\Gamma) \otimes L$,
with $L$ identified with $e \otimes L$, and $\pi = \lambda_{\Gamma} \otimes 
\mathrm{Id}_L$.

We use $L$ as a scale to measure the $\Gamma$ - invariant ("unbounded")
vectors in $H_0$ (this is exactly the Petersson scalar product).
By $P_L$ we denote the orthogonal projection onto the space $L$.

\vskip6pt

\begin{definition}
Let $H_0^{\Gamma}$ be the space of densely defined, $\Gamma$ - invariant, with
 $\Gamma$ - invariant domain, functionals on $H_0$. Let $%
H_0^{\Gamma}(L, \pi)$ be the linear subspace consisting of vectors $v \in H_0^{\Gamma}$, such
that the vector $P_Lv$ (defined by $\langle P_Lv, w \rangle = \langle v,
P_Lw \rangle_{H}$ for $w \in \mathrm{Dom}(v)$), is a bounded linear form on $%
H_0$. Thus, $H_0^{\Gamma}(L, \pi)$ will consist of vectors  $v$ such that $\langle P_Lv, v
\rangle_{H} < \infty$.

Noted that, in this case, $H_0^{\Gamma}(\pi, L)$ is a Hilbert space with scalar product $$\langle
v_1, v_2 \rangle_{H_0^{\Gamma}(\pi, L)} = \langle P_Lv_1, v_2 \rangle_{H}, v_1,v_2\in H_0^{\Gamma}(\pi, L).$$
\end{definition}

\vskip6pt

\begin{remark}
If the vector $v_2$ in $H_0^{\Gamma}$ has the property that there exists $%
w_2 \in H_0$ such that $v_2 = \mathop{\sum}\limits_{\gamma \in \Gamma}
\pi_0(\gamma)w_2$ (in the sense of pointwise convergence on a dense subset,
which is fulfilled, for example, if $w_2$ is a $\Gamma$ - wandering vector),
then 
\begin{equation*}
\langle v_1, v_2 \rangle_{H_0^{\Gamma}(\pi, L)} = \langle P_Lv_1, %
\mathop{\sum}\limits_{\gamma_2}\pi(\gamma_2)w_2 \rangle_{H} =
\end{equation*}
\begin{equation*}
= \langle \sum \pi(\gamma_2)P_Lv_1, v_2 \rangle_H = \sum \langle
\pi(\gamma_2)p\pi(\gamma_2^{-1})v_1, w_2 \rangle_H = \langle v_1, w_2
\rangle_H.
\end{equation*}

\

We consider the following linear map $\Phi$, defined on  $L$, onto $H_0^{\Gamma}(\pi, L)$, defined by the formula: 
\begin{equation*}
\Phi(l) = \mathop{\sum}\limits_{\gamma}\pi_0(\gamma)l = p(\sum \pi(\gamma)l).
\end{equation*}

It is well defined as $\sum \pi(\gamma)l$ defines densely defined $\Gamma$ -
invariant functionals on $H$.
\end{remark}

\vskip6pt

\begin{proposition}
Assume that the space $H_0^{\Gamma}(\pi, L)$ is finite dimensional and assume that $P_Lp$ is a trace
class operator. Let $\widetilde{\varepsilon}$ be the
restriction to $p(R(G) \otimes B(L))p$ of the unbounded character $%
\varepsilon \otimes \mathrm{Id}$, which is the unbounded character on $\mathcal R(G)\otimes B(L)$, obtained by summation over 
 $G$, of the operator coefficients belonging to $B(L)$.
Here $S_p$ is the representation constructed in Section 3.

 Let $$A(\Gamma\sigma\Gamma) = \mathop{\sum}\limits_{\theta \in
\Gamma\sigma\Gamma} P_L\pi_0(\theta)P_L, \ \ \sigma \in G.$$

Then $$A(\Gamma\sigma\Gamma) = \widetilde{\varepsilon}(S_p(\Gamma\sigma%
\Gamma)), \sigma \in G.$$ 

Moreover, the map $$%
\Gamma\sigma\Gamma\rightarrow A(\Gamma\sigma\Gamma),\ \  \sigma \in G,$$
extends by linearity to a  representation of the Hecke algebra. 
In addition $A(\Gamma)
= \widetilde{\varepsilon}(p)$ is a projection, and $A(\Gamma\sigma\Gamma) = 
\widetilde{\varepsilon}(p)A(\Gamma\sigma\Gamma)\widetilde{\varepsilon}(p)$.

Moreover $A(\Gamma) = \Phi^{\ast}\Phi$, and hence $A(\Gamma)$ is a finite
dimensional projection.
In addition,  if $\widetilde{H}_0^{\Gamma}(L, \pi)$ is the image of $\Phi$ in $%
H_0^{\Gamma}(L, \pi)$, then $\Phi$ intertwines the action of the Hecke
operators $T_{H_0}(\Gamma\sigma\Gamma)$ on $\widetilde{H}_0^{\Gamma}(L, \pi)$
with the representation $[\Gamma\sigma\Gamma] \to A(\Gamma\sigma\Gamma)$ into 
$B(\widetilde{\varepsilon}(p)L)$.

In particular 
\begin{equation*}
\mathrm{Tr}(T_{H_0}(\Gamma\sigma\Gamma)|_{\widetilde{H}_0^{\Gamma}(L, \pi)})
= \mathop{\sum}\limits_{\theta \in \Gamma\sigma\Gamma}\mathrm{Tr}%
(P_L\pi_0(\theta)P_L)
\end{equation*}

\noindent (which is the formula used in Zagier's paper ([Za]) on the Eichler
Selberg trace formula).
\end{proposition}

\vskip6pt

\begin{proof}

Clearly for $l_1, l_2 \in L$
$$
\langle \Phi(l_1), \Phi(l_2) \rangle = \langle \mathop{\sum}\limits_{\gamma_1} \pi_0(\gamma_1)l_1, \mathop{\sum}\limits_{\gamma_2} \pi_0(\gamma)l_2 \rangle_{H_0^{\Gamma}(L, \pi)}.
$$

This, as we observed in the preceding remark, is equal to
$$
\langle \mathop{\sum}\limits_{\gamma} \pi_0(\gamma)l_1, l_2 \rangle =  \mathop{\sum}\limits_{\gamma} \langle P_L\pi_0(\gamma)P_L l_1, l_2 \rangle.
$$

Thus $\Phi^{\ast}\Phi = \mathop{\sum}\limits_{\gamma} P_L\pi_0(\gamma)P_L$, but by our assumption the image of $\Phi$ is finite dimensional, thus the sum above is weakly convergent. The sane argument works for the similar sums defining $A(\Gamma\sigma\Gamma)$.

The intertwining formula for $\Phi$ is a consequence of the fact that the
action of the Hecke operator $T_{H_0}(\Gamma\sigma\Gamma)$ on a vector of the form 
$v = \sum \pi_0(\gamma)w$ is $T_{H_0}(\Gamma\sigma\Gamma)v = \mathop{\sum}%
\limits_{\theta \in \Gamma\sigma\Gamma} \pi_0(\theta)w$.

\end{proof}

\vskip6pt

\vskip6pt

\begin{remark}
Note that in the case of a representation $\pi_0$ coming from a unitary
representation $\overline{\pi}_0$ on $H_0$ of a larger group $\overline{G}$
which contains $G$ as a dense subgroup, and such that the representation
(projective) $\pi$ is in the discrete series of $G$, we have that 
\begin{equation*}
\mathrm{Tr}(P_Lp) = \mathrm{dim}_{\pi_0(\Gamma)}H_0 = \dfrac{\mathrm{card}\;
F}{d_\pi}
\end{equation*}

\noindent (while as observed in the previous proposition 
\begin{equation*}
\mathrm{dim}_{\mathbb{C}}H_0^{\Gamma}(L, \pi) = \mathop{\sum}\limits_{\gamma
\in \Gamma}\mathrm{Tr}(P_L\pi_0(\gamma))\; ).
\end{equation*}
\end{remark}

\vskip6pt

\begin{proof}

This is essentially the formula proven in [GHJ]. For completeness we reprove it here.

The Plancherel dimension formula gives that for $A_0$ a trace class operator in $\mathcal C_1(H_0)$, we have
$$
\mathop{\int}\limits_{G} \pi_0(g)A_0\pi_0(g) {\rm d}g = d_\pi {\rm tr}(A_0)p.
$$

Because of the arguments in [GHJ], we have that
$$
{\rm dim}_{\{\pi_0(\Gamma)\}''}H_0 =
{\rm Tr}(P_Lp) = {\rm Tr}(P_LpP_L) = {\rm Tr}(pP_Lp).
$$

But then
$$
d_\pi p{\rm Tr}(pP_lp) = \mathop{\int}\limits_{G}\pi_0(g)pP_lp\pi_0(g)^{-1}dg =
$$
$$
= \mathop{\int}\limits_{G/\Gamma} \pi_0(g)(\mathop{\sum}\limits_{\gamma\in \Gamma} \pi_0(\gamma)pP_Lp\pi_0(\gamma)^{-1})\pi_0(g)^{-1}dg.
$$

But $\mathop{\sum}\limits_{\gamma \in \Gamma}\pi_0(\gamma)pP_Lp\pi_0(\gamma)^{-1}= p$ and hence the integral is further equal to
$$
\mathop{\int}\limits_{G/\Gamma}\pi_0(g)p\pi_0(g)^{-1}dg = p\; ({\rm covol}(\Gamma)).
$$

\end{proof}

\vskip6pt

\begin{remark}
In the sum $\mathop{\sum}\limits_{\theta \in \Gamma\sigma\Gamma}\mathrm{Tr}%
(\pi_0(\theta)P_L)$, if we divide into  the conjugacy equivalence orbits $%
\Gamma\sigma\Gamma/\sim$ of the action of  $\Gamma$ on $\Gamma\sigma\Gamma$, we get 
\begin{equation*}
\mathop{\sum}\limits_{\theta \in \Gamma\sigma\Gamma/\sim} \ \mathop{\sum}%
\limits_{\gamma \in \Gamma} \mathrm{Tr}(\pi_0(\theta)\pi_0(\gamma)P_L\pi_0(%
\gamma)^{-1}).
\end{equation*}
The term, for $\theta \in G$, 
\begin{equation*}
\mathop{\sum}\limits_{\gamma \in \Gamma} \mathrm{Tr}(\pi_0(\theta)\pi_0(%
\gamma)P_L\pi_0(\gamma^{-1}))
\end{equation*}

\noindent is a conjugacy invariant, equal to the coefficient, by which the
conjugacy class of $\theta$ enters in the Plancherel formula for $\pi$. In
fact, in the case of $\overline{G} = PGL_2(\mathbb{R})$, by using the
Berezin ([Be]) symbol $\pi(\widehat{\theta})(\overline{z}, \eta)$ of the
unitary $\pi(\theta)$, this sum is equal to 
\begin{equation*}
\mathop{\sum}\limits_{\gamma\in \Gamma} \mathop{\int}\limits_{\gamma F} \widehat{%
\pi_0(\theta)}(\overline{z}, z) d\nu_0(z)
\end{equation*}

\noindent which is further equal to 
\begin{equation*}
\mathop{\int}\limits_{\mathbb{H}} \widehat{\pi_0(\theta)}(\overline{z}, z)
d\nu_0(z) = \mathop{\int}\limits_{\mathbb{H}} \dfrac{1}{(1 - \overline{z}%
\theta z)^t}(j(\theta, z))^t d\nu_0(\eta),
\end{equation*}

\noindent (this is the coefficient that shows up in Zagier's  formula in [Za]).
\end{remark}

\vskip6pt

\section{Realization of the
Hecke operators as values of a matrix valued scalar product}

In this section we prove that our construction provides a realization of the
Hecke operators as values of a matrix valued scalar product on the operators
system ${\mathcal{K}}_0{\mathcal{K}}_0^{\ast}$ described in 
Section 2, Definition \ref{os}.

	Recall that ${\mathcal{K}}_0^{\ast}{\mathcal{K}}_0 = \mathbb{C} (G\slash %
\Gamma) \mathop{\otimes} \limits_{{\mathcal{H}}_0}C (\Gamma\setminus G)$.
We define a scalar product on ${\mathcal{K}}_0{\mathcal{K}}_0^{\ast} = 
\mathbb{C} (G\setminus\Gamma) \otimes l^2(\Gamma\setminus G)$ with values in
the algebra of Hecke operators.
The scalar product is compatible with the ${\mathcal{H}}_0$ module structure
on ${\mathcal{K}}_0^{\ast}{\mathcal{K}}_0$.

\vskip6pt

\begin{definition}
Assume that $\pi_0, G, H_0$ is a representation of $G$ as in the previous
section. We make the additional assumption that $H_0$ splits as $l^2(\Gamma)
\otimes L_0$, where the Hilbert space $L_0$ is finite dimensional, and such
that $\pi_0|_{\Gamma}$ is unitarily equivalent to $\lambda_{\Gamma} \otimes 
\mathrm{Id}_{L_0}$.

Let $S_0 : {\mathcal{K}}_0^{\ast}{\mathcal{K}}_0 \supseteq {\mathcal{H}}_0
\to R(G) \otimes B(L_0)$ be the representation of the Hecke algebra and of the
larger operator system ${\mathcal{K}}_0^{\ast}{\mathcal{K}}_0$ constructed
in Section 2.

We will say that $\pi_0$ is "tamed" if the range of $S_0({\mathcal{K}}_0^{\ast}%
{\mathcal{K}}_0)$ is contained in the domain of the algebra homeomorphism $$\varepsilon \otimes \mathrm{%
Id}_{B(L_0)} : R(G) \otimes B(L_0) \to B(L_0).$$ In particular this implies
that the matrix coefficients of $$S_0^{\sigma_1\Gamma}S_0^{\Gamma\sigma_2} = %
\mathop{\sum}\limits_{\theta \in \sigma_1\Gamma\sigma_2} \rho(\theta^{-1})
\otimes P_{L_0}\pi_0(\theta)P_{L_0},$$ are in $l^1(\sigma_1\Gamma\sigma_2,
B(L))$, for $\sigma_1, \sigma_2 \in G$.
\

\

Consider the densely defined linear operator ${\mathcal{I}}$ on $R(G)$ with
values into $\{\rho_{G/\Gamma}(G)\}''$, defined by ${\mathcal{I}}(\sum a_g\rho_g) =
\sum a_g\rho_{G/\Gamma}(g)$, where $\rho_{G/\Gamma} : G \to B(l^2(G/\Gamma))$
is the right regular representation of $G$ into $l^2(G/\Gamma)$.

Clearly our hypothesis implies that the domain of ${\mathcal{I}}$ contains
the range of $S$.

Let $\widetilde{\varepsilon} : \{\rho_{G/\Gamma}(G)\}'' \to \mathbb{C}$ be the
densely defined character given by the formula $$\widetilde{\varepsilon}(\sum
a_g\rho_{G/\Gamma}(g)) = \sum a_g \in \mathbb{C}.$$ Then the range of ${%
\mathcal{I}} \circ S$ is contained in the domain of $\widetilde{\varepsilon}
\circ \mathrm{Id}_{B(L_0)}$, and we have the the following commutative
diagram 
\begin{equation*}
\begin{array}{cccccccc}
{\mathcal{K}}_0^{\ast}{\mathcal{K}}_0 & \mathop{\longrightarrow}\limits^{S_0}
& R(G) \otimes B(L_0) & \mathop{\longrightarrow}\limits^{{\mathcal{I}}
\otimes \mathrm{Id}_{B(L_0)}} & \{\rho_{G/\Gamma}(G)\}'' \otimes B(L) &  &  &  \\%
[3mm] 
& \mathop{\searrow}^{T_0} & \downarrow\; \varepsilon \otimes \mathrm{Id}%
_{B(L_0)} & \swarrow\; \widetilde{\varepsilon} &  &  &  &  \\[3mm] 
&  & B(L_0) &  &  &  &  & 
\end{array}%
\end{equation*}

Thus we obtain a representation $T_0 = (\varepsilon \circ \mathrm{Id}%
_{B(L_0)}) \circ S$ which extends the representation of Hecke operators from 
${\mathcal{H}}_0$ to ${\mathcal{K}}_0^{\ast}{\mathcal{K}}_0$. The image of
the representation $T_0$ consists of operators acting on the space of
vectors in $L_0$. This latest space is canonically identified to $\Gamma$ -
invariant vectors in $H_0$.
\end{definition}

\vskip6pt

The main reason for introducing the additional map ${\mathcal{I}}$ in the
above diagram, is because, to compute values of $T_0$ on $%
\sigma_1\Gamma\sigma_2$ we do a sum over $\sigma_1\Gamma\sigma_2$ which
corresponds to a matrix element in the factorization through $\{\rho_{G/\Gamma}(G)\}''$.

In fact 
\begin{equation*}
(\varepsilon \otimes \mathrm{Id}_{B(L_0)})(S(\sigma_1\Gamma\sigma_2)) = %
\mathop{\sum}\limits_{\theta \in
\sigma_1\Gamma\sigma_2}P_{L_0}\pi_0(\theta)P_{L_0}.
\end{equation*}

Indeed this follows from the fact that for $X = \sum a_g \rho_{G/\Gamma}(g)$%
, where $a_g$ are complex coefficients, $L^1(G)$, we have 
\begin{equation*}
\langle {\mathcal{I}}(X)\sigma_1\Gamma, \sigma_2\Gamma \rangle = %
\mathop{\sum}\limits_{g \in \sigma_2\Gamma\sigma_1^{-1}} a_g
\end{equation*}

\noindent and hence if $X = \mathop{\sum}\limits_{g \in
\sigma_4\Gamma\sigma_3^{-1}} a_g\rho_{G/\Gamma}(g)$ then $\langle
X\sigma_1\Gamma, \sigma_2\Gamma \rangle = \mathop{\sum}\limits_{g \in
\sigma_4\Gamma\sigma_3^{-1} \cap \sigma_2\Gamma\sigma_1^{-1}} a_g$.

\vskip6pt

\begin{proposition}
There exists a $B(L_0)$ valued, scalar product on $${\mathcal{K}}_0{\mathcal{K}%
}_0^{\ast} = l^2(G\setminus \Gamma) \otimes l^2(\Gamma \setminus G),$$ with
the following properties

1) $\langle \Gamma\sigma_1 \otimes \sigma_1\Gamma, \Gamma\sigma_2 \otimes
\sigma_2\Gamma \rangle = T_0(\sigma_1\Gamma)T(\Gamma\sigma_2) =
T_0(\Gamma\sigma_1)^{\ast}T_0(\Gamma\sigma_2)$, \; $\sigma_1, \sigma_2 \in G$%
. In particular if we take a disjoint reunion of cosets of the form $%
\sigma_1\Gamma, \Gamma\sigma_2$ whose reunion is a double coset, we get the
the Hecke operator corresponding to the double coset.

2) $\langle \Gamma\sigma_1 \otimes \sigma_2\Gamma, \Gamma\sigma_3 \otimes
\sigma_4\Gamma \rangle$ depends only on $\sigma_3\Gamma\sigma_1 \cap
\sigma_4\Gamma\sigma_2$, for all $\sigma_1, \sigma_2, \sigma_3, \sigma_4 \in
G$.

Equivalently by cyclically rotating the variables we obtain a bilinear form $%
\ll \cdot, \cdot \gg$ on ${\mathcal{K}}_0^{\ast}{\mathcal{K}}_0 =
l^2(\Gamma\setminus G) \mathop{\otimes}\limits_{{\mathcal{H}}} l^2(G/\Gamma)$%
, defined by the formula 
\begin{equation*}
\ll [\sigma_3\Gamma] \mathop{\otimes}\limits_{{\mathcal{H}}_0}
[\Gamma\sigma_1], [\sigma_4\Gamma] \mathop{\otimes}\limits_{{\mathcal{H}}_0}
[\Gamma\sigma_2] \gg\; = \langle \Gamma\sigma_1 \otimes \sigma_2\Gamma,
\Gamma\sigma_3 \otimes \sigma_4\Gamma \rangle.
\end{equation*}

An equivalent form to describe this property of the scalar product is the property that
for every $r \in {\mathcal{H}}_0$, we have 
\begin{equation*}
\langle \Gamma\sigma_1 \otimes [(\sigma_2\Gamma)r], \Gamma\sigma_3 \otimes
\sigma_4\Gamma \rangle = \langle \Gamma\sigma_1 \otimes \sigma_2\Gamma,
\Gamma\sigma_3 \otimes [(\sigma_4\Gamma)r] \rangle
\end{equation*}

\noindent for all $\sigma_1, \sigma_2, \sigma_3, \sigma_4 \in G$.

3) The support properties are preserved by the scalar product: For all $%
\sigma_1, \sigma_2, \sigma_3, \sigma_4 \in G$ if $\Gamma\sigma_1%
\sigma_2^{-1}\Gamma \cap \Gamma\sigma_3^{-1}\sigma_4\Gamma$ is void, or if $%
\sigma_3\Gamma\sigma_1 \cap \sigma_4\Gamma\sigma_2$ is void, then the scalar product $$%
\langle \Gamma\sigma_1 \otimes \sigma_2\Gamma, \Gamma\sigma_3 \otimes
\sigma_4\Gamma \rangle ,$$
vanishes.
\end{proposition}

\vskip6pt

\begin{proof}

With the previous notations, let $\widetilde{S}_0$ be the composition $\I \circ S_0$ which is a multiplicative map form $\K_0^{\ast}\K_0$ into $B(L_0)$.

By the  multiplicativity property we understand the fact that
$$
\widetilde{S}_0(k_1^{\ast})\widetilde{S}_0(k_2) = \widetilde{S}_0(k_1)^{\ast}\widetilde{S}_0(k_2) = \widetilde{S}_0(k_1^{\ast}k_2)
$$

\noindent for all $k_1, k_2 \in \K$.

Note that $\widetilde{S}_0(k)$ belongs to $\{\rho_{G/\Gamma}(G)\}''$ which, by the results in  [BC], [Tz], is the commutant in $B(l^2(G\setminus \Gamma))$ of the Hecke algebra $\H_r$, acting from the right on $l^2(G\setminus \Gamma)$.

For $h, k$ vectors in $l^2(G/\Gamma)$ we let
$$
\widetilde{w}_{h,k} : B(l^2(G/\Gamma) \otimes B(L_0)) \to B(L_0)
$$

\noindent be the functional $\langle \cdot, h\rangle k \otimes {\rm Id}_{B(L_0)}$.

For simplicity for $h, k$ in $l^2(G/\Gamma)$ we denote $\widetilde{w}_{h,k}(A) = \langle Ah, k \rangle$. Note that this is an element in $B(L_0)$.

Then, $\langle \widetilde{S}(\Gamma\sigma_1)\sigma_2\Gamma, \widetilde{S}(\Gamma\sigma_3)\sigma_4\Gamma \rangle$, with the above notation, is an element in $B(L_0)$.

Moreover because $\widetilde{S}$ is a representation of $\K_0^{\ast}\K_0$, this is further equal to
$$
\langle \widetilde{S}(\sigma_3^{-1}\Gamma\sigma_1)\sigma_2\Gamma, \sigma_4\Gamma \rangle.
$$

By the previous remark, this is equal to
$$
\mathop{\sum}\limits_{\theta \in \sigma_3^{-1}\Gamma\sigma_1 \cap \sigma_4\Gamma\sigma_2^{-1}} P_{L_0}\pi_0(\theta)P_{L_0}.
$$

We define then the scalar product by the formula
$$
\langle \Gamma\sigma_1 \otimes \sigma_2\Gamma, \Gamma\sigma_3 \otimes \sigma_4\Gamma = \langle \widetilde{S}_0(\Gamma\sigma_1)[\sigma_2^{-1}\Gamma], \widetilde{S}_0(\Gamma\sigma_3^{-1})(\sigma_4\Gamma) \rangle.
$$

This formula also proves the positive definiteness of the scalar product.

This is consequently equal to
$$
\mathop{\sum}\limits_{\theta \in \sigma_3\Gamma\sigma_1 \cap \sigma_4\Gamma\sigma_2} P_{L_0}\pi_0(\theta)P_{L_0}. \eqno(\ast\ast)
$$

This proves property (2). 

Property (1) now follows from the fact that
$$
\langle \Gamma\sigma_1 \otimes \sigma_1\Gamma, \Gamma\sigma_2 \otimes \sigma_2\Gamma \rangle
$$

\noindent is equal, by formula ($\ast\ast$) to
$$
\mathop{\sum}\limits_{\theta \in \sigma_1\Gamma\sigma_2} P_{L_0}\pi_0(\theta)P_{L_0}
$$

\noindent which by the previous diagram is
$$
T_0(\sigma_1\Gamma)T_0(\Gamma\sigma_2) = T_0(\Gamma\sigma_1)^{\ast}T_0(\Gamma\sigma_2)
$$

\noindent for all $\sigma_1, \sigma_2 \in G$.

\end{proof}

\vskip6pt

\begin{remark}
The operators $T_0(\sigma_1\Gamma)T_0(\Gamma\sigma_2)$ are not classical
Hecke operators $T_0$ acting on $L_0$, but for all $\sigma \in G$, we have, by the multiplicativity
of the representation $T_0 : {\mathcal{K}}_0^{\ast}{\mathcal{K}}_0 \to
B(L_0) $, that for all $\sigma_1, \sigma_2 \in G$, $\Gamma\sigma_1\Gamma =
\cup s_i \sigma_1\Gamma$ the following equality:
\begin{equation*}
\mathop{\sum}\limits_{s_i} T_0(s_i\sigma_1\Gamma)T_0(\Gamma\sigma_2) =
T_0(\Gamma\sigma_1\Gamma)T_0(\Gamma\sigma_2),
\end{equation*}
holds true.
\end{remark}

\vskip6pt

\section{An explicit formula for the absolute value of the
matrix coefficients  of $\pi_{13}$, the 13-th element in the
discrete series of $PSL_2(\mathbb{R})$}

In this section we find an explicit formula for the absolute value of the
matrix coefficients $t_{13}$ of $\pi_{13}$, the 13-th element in the
discrete series of $PSL_2(\mathbb{R})$, evaluated on a specific choice of $%
\Gamma$ - wandering vector. In particular, this proves that, for this
representation, we may arrange so that the elements $T^{\Gamma\sigma\Gamma}$, constructed in the previous section, which a
priori belong to $l^2(\Gamma\sigma\Gamma) \cap {\mathcal{L}}(G)$, also
belong to the domain of $\varepsilon$.

 In fact, we prove that the coefficients $%
|t_{13}(\theta)|^2$, $\theta \in G$, are of the form $\mu_0(\theta F\cap F)$,  where $F$ is a fundamental
domain for the action of $\Gamma$, in the upper half plane $\mathbb{H}$. Here, the choice of
the fundamental domain corresponds to a choice of the $\Gamma$ - wandering
vector in $H_{13}$, which is unique up to a unitary in $\{ \pi_{13}(\Gamma)
\}^{\prime}\cong R(\Gamma) \cong {\mathcal{L}}(\Gamma)$). Recall that $\mu_0$ is
the canonical $PSL_2(\Bbb R)$ - invariant measure on the upper-half plane $\mathbb{%
H}$.

First, we note the properties of the positive definite function we are
identifying:

\vskip6pt

\begin{proposition}
Let $\Gamma \subseteq G$ be a discrete group and let $\pi$ be a
representation of $G$, on a Hilbert space H, such that $\pi|_{\Gamma}$ is
unitarily equivalent to the left regular representation of $\Gamma$ (Thus $H
\cong l^2(\Gamma)$) and there exists $\eta$ in $H$ such that $%
\pi(\gamma)\eta $ is orthogonal to $\eta$ for all $\gamma \neq e$ and $%
\overline{Sp\{\pi(\gamma)\eta | \gamma \in \Gamma\}} = H$.

Let $t(\theta) = \langle \pi(\theta)\eta, \eta \rangle$, $\theta \in G$.
Then $t(\theta)$ is a positive definite function on $G$.
 
In the setting introduced in Section 2, $t(\theta)$ is $\mathrm{Tr}(\pi(\theta)p_{\eta})$%
, where $p_{\eta}$ is the projection onto the 1-dimensional space generated
by $\eta$. Then $t$ has the following properties:

1) $t(\gamma) = \delta_{\gamma, e}$, $\gamma \in \Gamma$, where $e$
is the identity element of $G$ and $\delta$ is the Kronecker symbol.

2) $t$ is a positive definite function on $G$.

3) $\mathop{\sum}\limits_{\gamma \in \Gamma}|t(\gamma g)|^2 = 1$.

In addition $t^{\Gamma\sigma\Gamma} = \mathop{\sum}\limits_{\theta \in
\Gamma\sigma\Gamma} t(\theta)\rho(\theta)$ is a representation of ${\mathcal{%
H}}$ into $R(G)$ (this was proved in Section 2).

Note that if $\pi$ extends to a larger, continuous group $\overline{G}$
which contains $G$, then property (3) holds on $\overline{G}$ also.
\end{proposition}

\vskip6pt

\begin{proof}

The only fact that wasn't observed in Section 2 is property (3). But this follows from the fact that $\pi(g)\eta$ belongs to $\overline{S_p\{\pi(\gamma)\eta | \gamma \in \Gamma\}}$ and because $\pi(\gamma)\eta$, $\gamma \in \Gamma$, is an orthonormal basis.

\end{proof}

\vskip6pt

\begin{remark}
\label{phi} In the previous setting, if $\varphi_0(g) = |t(g)|^2 \geq 0$, $g
\in \overline{G}$ then $\varphi_0$ has exactly the behavior of a state of
the form $\mu(gF \cap F)$, $g \in \overline{G}$, where $\overline{G}$ acts
on a measure space $(X, \mu)$, by measure preserving transformations, and $F$ is a fundamental domain in $X$ for the action of $\Gamma$. Then,
the properties (1), (3) correspond to the fact that $F$ is a fundamental
domain for the action of  group $\Gamma$ on $X$ (and because of the implicit condition that $G(F) = \Gamma(F))$.

Note that the vector $\eta$ is not unique, in fact all others vectors with
similar properties to $\eta$, are of the form $u\eta$ where $u$ is a unitary
in $\{\pi(\Gamma)\}^{\prime\prime}\cong R(\Gamma)$.
Thus the other equivalent forms of the representation of the Hecke algebra ${%
\mathcal{H}}$ given by $\Gamma\sigma\Gamma\rightarrow t^{\Gamma\sigma\Gamma}$
are of the form $\Gamma\sigma\Gamma\rightarrow
u^{\ast}t^{\Gamma\sigma\Gamma}u$, $\sigma \in G$ for a unitaries $u$ in ${%
\mathcal{U}}(R(\Gamma))$.
\end{remark}

\vskip6pt

In the rest of the section $G = PGL_2(\mathbb{Z}[\frac{1}{p}])$, with $p$ a
prime number (or $G=PGL_2(\mathbb{Q})$), $\Gamma = PSL_2(\mathbb{Z})$ and $%
\overline{G} = PSL_2(\mathbb{R})$. The representation $\pi$ will be the
13-th element in the discrete series of $PSL(2, \mathbb{R)}$ (a projective,
unitary representation). The representation $\pi_{13}$ acts on the Bargman
Hilbert space $H_{13}$ (see [Be], [Ra] for definitions). Recall $H_{13} =
H^2(\mathbb{H}, ({\rm Im\ }z)^{13 -2} d\overline{z}dz)$. By $\mu_0$ we denote  the standard $G$ invariant measure on $\mathbb{H}$.

We identify $H_{13} \otimes \overline{H}_{13}$ with ${\mathcal{C}}_2(H_{13})$%
. A typical element of ${\mathcal{C}}_2(H_{13})$ is the Toeplitz operator $%
T_f^{13}$, which in the sequel we denote by $T_f$ simply, where $f$ is a
function of compact support in $\mathbb{H}$.

Recall that by (Be]), there exists a bounded operator $B \geq 0$, the
Berezin's transform operator, with zero kernel, commuting with $\pi(g)$, $g
\in \overline{G}$. Moreover,
$B$ is a function of the invariant Laplacian $\dfrac{\partial^2}{\partial
z\partial \overline{z}}({\rm Im\ } z)^{-2}$ such that 
\begin{equation*}
\mathrm{Tr}(T_fT_g^{\ast}) = \langle B^{-1}f, g \rangle_{L^2(\mathbb{H},
\mu_0)}, \ \ f,g\in L^2(\mathbb{H},
\mu_0).
\end{equation*}

Here we are  identifying the vector $\eta$ (or rather) $\eta \otimes 
\overline{\eta}$ in $H_{13} \otimes \overline{H}_{13}$ such that $$%
|t(\theta)|^2 = \langle (\pi_{13} \otimes \overline{\pi}_{13})\eta \otimes 
\overline{\eta}, \eta \otimes \overline{\eta} \rangle, \theta \in G.$$
In the identification $H_{13} \otimes \overline{H}_{13}$ with ${\mathcal{C}}%
_2(H_{13})$ (the Schatten von Neumann ideal of square summable operators on $%
H_{13}$) the representation $\pi_{13}$ becomes $\mathrm{Ad}\pi_{13}$ on ${%
\mathcal{C}}_2(H_{13})$, and $\eta \otimes \overline{\eta}$ becomes the
projection $p_{\eta}$ onto $\mathbb{C}\eta$. Then $$|t(\theta)|^2 = \mathrm{Tr%
}(p_{\eta}(\mathrm{Ad\pi_{13}(\theta)})(p_{\eta})),$$ while $$t(\theta) = 
\mathrm{Tr}(\pi_{13}(\theta)p_{\eta}), \theta \in G.$$

Let ${\mathcal{C}}_{13}$ be the closed selfdual cone in ${\mathcal{C}}%
_2(H_{13})$ generated by projections $p_{\zeta}$, $\zeta \in H_{13}$, (that
is ${\mathcal{C}}_{13}$ is the cone of positive elements). Then ${\mathcal{C}%
}_{13}$ is invariated by the representation $\mathrm{Ad}\pi_{13}$. \vskip6pt

\begin{remark}
Because of the above formula for the scalar product in ${\mathcal{C}}%
_2(H_{13})$ of two Toeplitz operators, we have that we have a canonical
unitary $U$ between $L^2(\mathbb{H}, \mu_0)$, and ${\mathcal{C}}_2(H_{13})$,
mapping functions $f$ in $L^2(\mathbb{H}, \mu_0)$ into the Toeplitz operator 
$T_{B^{-1/2}f}$. Moreover the unitary $U$ intertwines the corresponding
representations of $PSL(2,R)$ and hence the selfdual closed cone ${\mathcal{C%
}}_{13}$ corresponds to the positive functions in $L^2(\mathbb{H}, \mu_0)$.
\end{remark}

\vskip6pt

\begin{proof}

Indeed since $[\pi_{13}(g), B] = 0$ this is the only  selfdual cone invariated by ${\rm Ad}\pi_{13}(g)$, $g \in G$.

\end{proof}

\vskip6pt

\begin{proposition}
Let $F$ be a fundamental domain for the action of $\Gamma$ on $\mathbb{H}$.
Let $f_0 = B^{1/2}\chi_F$. Then $T_{f_0}$ is of the form $p_{\eta}$, where $%
\eta$ is a $\Gamma$-wandering vector in $H_{13}$.

In particular the positive definite functions $%
\varphi_0$, introduced in Remark \ref{phi}, are of the form 
\begin{equation*}
\varphi_0(g) = \mu_0(F \cap gF), \; g \in PSL_2(\mathbb{R}).
\end{equation*}
\end{proposition}

\vskip6pt

\begin{proof}

Indeed
$$
{\rm Tr}(\pi_{13}(\gamma)T_{f_0}\pi_{13}(\gamma)^{-1}T_{f_0}) = {\rm Tr}(T_{B^{1/2}\chi_{\gamma F}}T_{B^{1/2}\chi_F}) =
$$
$$
= \langle B^{-1}B^{1/2}\chi_{\gamma F}, B^{1/2}\chi_F \rangle = 0
$$

\noindent for $\gamma \in \Gamma \setminus \{e\}$.

Moreover $$\mathop{\sum}\limits_{\gamma\in \Gamma} \pi(\gamma)T_{f_0}\pi(\gamma)^{-1} = \sum\limits_{\gamma\in\Gamma} T_{B^{-1/2}\chi_{\gamma F}},$$ which is a multiple of the identity operator.

Thus ${\rm Tr}(T_{f_0})$ computes the trace either on $\{\pi_{13}(\Gamma)\}''$ and on $\{\pi_{13}(\Gamma)\}'$. Consequently it since $T_{f_0}$ is positive (by the previous remark), it follows that the operators
 $$\pi_{13}(\gamma)T_{f_0}\pi_{13}(\gamma)^{-1},\  \gamma \in \Gamma,$$ have orthogonal ranges. Since we have the Murray von Neumann dimension dim$_{\Gamma} H_{13}$ is equal to 1,  it follows that the operator ${\rm Tr}(T_{f_0})$ is of the form $p_{\eta}$, where $\eta$ a $\Gamma$-wandering and cyclic vector in $H_{13}$.

\end{proof}

We make the following conjecture. If, in general,  $G$ is a semi-simple Lie group, and $\Gamma$
a lattice in $G$, let  $K_{G,\Gamma}$ be the convex set of continuous positive
definite functions $\phi$ on $G$ with the following the properties

1) $\phi (g) \geq 0$, for all $g$ in $G$ and $\phi$ is positive definite,

2) $\sum_{\gamma\in \Gamma}\phi(\gamma g)=1$, all $g$ in G,

3) $\phi(\gamma) =\delta_{\gamma, e}$, for $\gamma \in \Gamma$.

We conjecture that the  extremal points of Ext $K_{G,\Gamma}$, are the positive definite functions of
the form $\phi_F(g)=\mu_G(gF\cap F)$, where $F$ is a fundamental domain for
the action of $\Gamma$ on $G$ and $\mu_G$ is the left Haar measure on $G$,
normalized so that $\mu_G(F)=1$.

If $G$ admits a unitary (projective) discrete series representation $\pi$ on
a Hilbert space $H$ such that $\pi|_{\Gamma}$ is unitarily equivalent to the
left regular representation of $\Gamma$, then we conjecture that Ext $%
K_{G,\Gamma}$ coincides with the positive definite function of the form 
\begin{equation*}
g\rightarrow |\langle \pi(g)\eta, \eta\rangle|^2, g\in G,
\end{equation*}
where the vectors  $\eta$ run over vectors in $H$ that are trace vectors for $%
\pi(\Gamma) $.

\vskip6pt

\begin{corollary}
\label{calc}

In the previous context, we have that for all $\theta \in G$, the following equalities hold true 
\begin{equation*}
\varphi_0(\theta) = |t(\theta)|^2 = \mathrm{Tr}(T_{f_0}\pi_{13}(%
\theta)T_{f_0}\pi_{13}(\theta)^{-1}) =
\end{equation*}
\begin{equation*}
= \langle B^{-1}B^{1/2}\chi_F, B^{1/2}\chi_{\theta F} \rangle_{L^2(\mathbb{H}%
, \mu_0)} = \mu_0(\theta F \cap F).
\end{equation*}

Also $$t(\theta) = \mathrm{tr}(\pi_{13}(\theta)T_{B^{1/2}\chi_F}), \ \ \theta
\in G.$$

\end{corollary}

\vskip6pt

We conjecture that $T_{B^{1/2}\chi_F}$ is the element $\zeta^{-1/2}T_{%
\chi_F}\zeta^{-1/2}$ where $\zeta$ is the square root of the positive
definite function on $\Gamma $ defined by the formula $\gamma\rightarrow 
\mathrm{Tr}(\pi_{13}(\gamma)\chi_F), \gamma \in \Gamma$ (this is the element 
$E_{R(G) \otimes 1}(p)$ from Proposition \ref{ep}). By $\zeta^{-1/2}$ we
mean, in fact, the operator $\pi_{13}(\zeta^{-1/2})$, the  image of  $\zeta^{-1/2}$ through, the
representation $\pi_{13}$.

\begin{corollary}
With the above notations, if $F$ is a fundamental domain for $\Gamma$ and $%
\phi_0(\theta)=\mu_0(\theta F \cap F), \theta \in G$, then $\phi_0$ defines
a positive definite function $\psi_0$ on ${\mathcal{K}}^\ast_0{\mathcal{K}}%
_0 $, by the formula 
\begin{equation*}
\psi_0 (\sigma_1\Gamma,\Gamma \sigma_2)= \sum _{\gamma \in
\Gamma}\mu_0(\sigma_1\gamma\sigma_2 F\cap F)=\mu_0(\sigma_1\Gamma\sigma_2
F\cap F) , \sigma_1, \sigma_2 \in G.
\end{equation*}

Here, being positive definite,  means that for all $k_1,k_2,...,k_n \in {\mathcal{K}}_0$ the
matrix 
\begin{equation*}
(\psi_0 (k^\ast_i, k_j))_{i,j=1,2,...,n}
\end{equation*}
is positive definite. Since $\psi_0|_{{\mathcal{H}}_0}$ is the character
corresponding to the identity on ${\mathcal{H}}_0$, it follows that $\psi_0$ is a
positive extension of this character to ${\mathcal{K}}^\ast_0{\mathcal{K}}_0$.

Note that $\psi_0 (\sigma_1\Gamma,\Gamma \sigma_2)$ may be interpreted as
the matrix coefficient of the representation of $\sigma_1$ on $L^2(\mathbb{%
H}, \mu_0)$ computed  at the $\Gamma$ equivariant vectors $\chi_{\sigma_2\Gamma F}$
and $\chi_{\Gamma F}=\chi_{\mathbb{H}}$. Here $\chi_F$ is a cyclic
vector for the representation of $G$ on $L^2(\mathbb{H}, \mu_0)$. Hence $\phi_0$ contains all the information about the representation of the group $G$,
including the action on $\Gamma$ invariant vectors.
\end{corollary}

\begin{proof} Proof because of Proposition \ref {calc}, we have, with the notations from ([Ra]), that  for
 all $\sigma_1,\sigma_2 \in G$,
$$\psi_0 (\sigma_1\Gamma,\Gamma \sigma_2)=
\sum_{\theta\in \sigma_1\Gamma \sigma_2}|t(\theta)|^2=
\tau ( (t^{\sigma_1\Gamma}t^{\Gamma \sigma_2})^{\ast}(t^{\sigma_1\Gamma}t^{\Gamma \sigma_2}).$$
Since $\tau$ is the trace on the group algebra of $G$, this is
further equal to
$$\tau (t^{\sigma_1^{-1}\Gamma\sigma_1} t^{\sigma_2^{-1}\Gamma\sigma_2}).$$

This implies the positivity of $\psi_0$

\end{proof}
\vskip6pt

Note that there exists a construction of Merel [Me] (we are indebted to Alex
Popa for bringing to our attention this paper), of a finitely supported
element $X$ in $\mathbb{C}(\Gamma\sigma_p\Gamma)$, such that $%
\Gamma\sigma_p\Gamma \to X$ extends to a representation of the Hecke algebra
($p$ is a prime number).

Unfortunately, for the element $X$ constructed in the paper ([Me]), the
property that the representation of the Hecke algebra may be extended to ${%
\mathcal{K}}_0^{\ast}{\mathcal{K}}_0$ does not hold true. The extension
property is, as proved in ([Ra]), the main ingredient that allows to
represent Hecke operators on Maass forms, by using the above representation
of the Hecke algebra (see also the construction in Theorem \ref{t}).

The following lemma gives a characterization of the elements $X$ in $%
l^2(\Gamma\sigma_p\Gamma) \cap {\mathcal{L}}(G)$, that generate a
representation of ${\mathcal{K}}^{\ast}{\mathcal{K}}$ into ${\mathcal{L}}(G)$%
, coming from a representation $\pi$ of $G$, which has the property that $%
\pi|_{\Gamma}$ is unitarily equivalent to the left regular representation of 
$\Gamma$.

\vskip6pt

\begin{proposition}
Let $X$ be a selfadjoint element in $l^2(\Gamma\sigma_p\Gamma)$. Let $\chi_1$
be the radial element in the group algebra of $F_{\frac{p+1}{2}}$ (the free
group with $\dfrac{p+1}{2}$ generators) (thus $\chi_n = \mathop{\sum}%
\limits_{w}w$, where $w$ runs over words of length $n$, $n \in \mathbb{N}$).
Let $\tau_p$ be the trace on ${\mathcal{L}}(F_{\frac{p+1}{2}})$.

(1) Then, the correspondence $(\Gamma\sigma_p\Gamma)^n \to X^n$ extends to a representation of
the Hecke algebra if and only if ${\tau}_{{\mathcal{L}}(G)}(X^n) =
\tau_p(\chi_1^n)$.

(2) Let $X_n$ be the projection of $X^n$ onto $l^2(\Gamma\sigma_{p^n}\Gamma)$%
. Let $X_{\Gamma\sigma_{p^e}s}$ (respectively $X_{s\sigma_{p^e}\Gamma}$),
for $s \in G$, $e \in \mathbb{N}$, be the projection of $X_e$ onto $%
l^2(\Gamma\sigma_{p^e}s)$ (respectively $l^2(s\sigma_{p^e}\Gamma)$).

Then, the element $X$ is associated to a  representation $\pi$ of $G$ on a Hilbert space $H$ (with $\mathrm{dim}%
_{\Gamma}H =1$) as in Section 2 (or [Ra]), if and only if $%
(X_{\Gamma\sigma})^{\ast} = (X_{\sigma\Gamma})$, $\sigma \in G$ and $%
[\sigma_1\Gamma][\Gamma\sigma_2] \to X_{\sigma_1\Gamma}X_{\Gamma\sigma_2}$, $%
\sigma_1, \sigma_2 \in G$ extends to a representation of ${\mathcal{K}}%
_0^{\ast}{\mathcal{K}}_0$. Being a representation of ${\mathcal{K}}%
_0^{\ast}{\mathcal{K}}_0$ means that  $X_{\sigma_1\Gamma}X_{\Gamma\sigma_2} =
X_{\sigma_1\Gamma\sigma_2}$, where $X_{\sigma_1\Gamma\sigma_2}$ is the
projection onto $l^2(\sigma_1\Gamma\sigma_2)$ of $X_{\sigma_1\Gamma}X_{%
\Gamma\sigma_2\Gamma}$ or of $X_{\Gamma\sigma_1\Gamma}X_{\Gamma\sigma_2}$.
(Equivalently this means that $X_{\sigma_1\Gamma\sigma_2}$ is the projection onto $%
l^2(\sigma_1\Gamma\sigma_2)$ of $\mathop{\sum}\limits_{e \in {\mathcal{I}}}
X_e$, where $\mathop{\bigcup}\limits_{e \in {\mathcal{I}}}
[\Gamma\sigma_{p^e}\Gamma]$ covers $\sigma_1\Gamma\sigma_2$; this should hold
true for all $\sigma_1, \sigma_2 \in G$).

If these conditions are verified, there $\Psi(a) = E_{{\mathcal{L}}%
(\Gamma)}^{{\mathcal{L}}(G)}(XaX)$, $a \in {\mathcal{L}}(\Gamma)$ defines a
representation of the Hecke algebra into the algebra consisting of
completely positive maps on ${\mathcal{L}}(\Gamma)$.

(3) The condition in (2) is verified if the following weaker condition holds
true 
\begin{equation*}
X_{\Gamma\sigma_1\Gamma}X_{\Gamma\sigma_2} = \mathop{\sum}\limits_{r_j}
X_{\Gamma\sigma_1r_j\sigma_2}
\end{equation*}

\noindent where $\Gamma\sigma_1\Gamma$ is the disjoint union of $%
\Gamma\sigma_1r_j$. This property should be verified for all $\sigma_1,
\sigma_2$. Thus to verify condition (2) it is sufficient to verify that $X$
(and $X_n, \; n \geq 1$) act on $Sp\{ X_{\Gamma\sigma} \mid \sigma \in G \}$
exactly as $[\Gamma\sigma_{p^n}\Gamma]$ acts on $l^2(\Gamma/G)$.
\end{proposition}

\vskip6pt

\begin{rem}
The condition (2) automatically implies that $(X_{\Gamma\sigma})^{\ast}X_{%
\Gamma\sigma}$ is equal to $X_{\sigma\Gamma\sigma}$, where only component in 
$\Gamma$ is $\tau_e$, and thus $\tau_{{\mathcal{L}}(G)}((X_{\Gamma\sigma})^{%
\ast}X_{\Gamma\sigma})$ is 1, thus $\| X_{\Gamma\sigma} \|^2_{2, {\mathcal{L}%
}(G)} = 1$.
\end{rem}

\vskip6pt

\begin{proof}
To prove (1) observe that the fact that $\tau_{\L(G)}(X^n) = \tau(\chi_1^n)$, for all $n \in \N$, is equivalent to the fact $X^n$ decomposes as a linear space of $\{ X_k \mid k \leq n \}$ exactly as $\chi_1^n$ decomposes as a linear space of $\{ \chi_k \mid k \leq n \}$ (i.e. with the same coefficients).

Hence $\chi_n \to X_n$ is a $\ast$ isomorphism but this implies that $\Gamma\sigma_{p^n}\Gamma \to X_n$ is a $\ast$ isomorphism, which because of orthogonality extends to the weak closure of the algebras involved.

To prove (2), we only have to prove the converse (since the direct implication was proved in [Ra], and reproved in Section 2). If condition (2) is verified, then one defines a representation $\pi$ of $G$ on $l^2(\Gamma)$ simply by defining $\pi(\sigma)1$ (where 1 is the identity of the group $\Gamma$, viewed as an element in $l^2(\Gamma)$) to be $\sigma(t^{\Gamma\sigma})^{\ast}$. This then, by requiring that $\pi|_{\Gamma} = \lambda_{\Gamma}$, defines the entire representation $\pi$.

The fact that (3) implies (2) is a simple consequence of supports computations. Indeed, if we take $\sigma_1, \sigma_2 \in G$ then if $\Gamma\sigma_1\Gamma = \cup s_i\sigma_1\Gamma$ (as a disjoint union), then the cosets $(s_i\sigma_1)\Gamma\sigma_2$ are disjoint and their union is the union of the cosets in $[\Gamma\sigma_1\Gamma][\Gamma\sigma_2]$.

\end{proof}

\end{document}